\documentclass[11pt]{amsart}%
\usepackage{palatino}
\usepackage{mathpazo}
\usepackage{amsfonts}
\usepackage{amsmath}
\usepackage{amssymb,latexsym}
\usepackage{graphicx}
\usepackage[mathscr]{eucal}
\usepackage{amssymb}%
\providecommand{\U}[1]{\protect \rule{.1in}{.1in}}

\newtheorem{theorem}{Theorem}[section]

\newtheorem{definition}[theorem]{Definition}
\newtheorem{lemma}[theorem]{Lemma}

\theoremstyle{remark}
\newtheorem{remark}[theorem]{Remark}
\newtheorem{example}[theorem]{Example}

\numberwithin{equation}{section}

\begin{document}
\title[A universal law] {Monodromy of
generalized Lam\'{e} equations with
 Darboux-Treibich-Verdier potentials:\\ A universal law}
\author{Zhijie Chen}
\address{Department of Mathematical Sciences, Yau Mathematical Sciences Center,
Tsinghua University, Beijing, 100084, China }
\email{zjchen2016@tsinghua.edu.cn}
\author{Chang-Shou Lin}
\address{Department of Mathematics, Taiwan University, Taipei 10617, Taiwan }
\email{cslin@math.ntu.edu.tw}

\begin{abstract} The Darboux-Treibich-Verdier (DTV) potential $\sum_{k=0}^{3}n_{k}(n_{k}+1)\wp(z+\tfrac{
\omega_{k}}{2};\tau)$ is well-known as doubly-periodic solutions of the stationary KdV hierarchy (Treibich-Verdier, Duke Math. J. {\bf 68} (1992), 217-236). In this paper,
we study the
generalized Lam\'{e} equation with the DTV potential
\begin{equation*}
y^{\prime \prime }(z)=\bigg[  \sum_{k=0}^{3}n_{k}(n_{k}+1)\wp(z+\tfrac{
\omega_{k}}{2};\tau)+B\bigg]  y(z),\quad n_{k}\in \mathbb{N}
\end{equation*}
from the monodromy aspect.
We prove that the map from $(\tau, B)$ to the monodromy data $(r,s)$ satisfies a surprising universal law
$d\tau\wedge dB\equiv8\pi^2 dr\wedge ds.$  Our proof applies Panlev\'{e} VI equation and modular forms.
We also give applications to the algebraic multiplicity of (anti)periodic eigenvalues for the associated Hill operator.
\end{abstract}

\maketitle

\section{Introduction}

Throughout the paper, we use the notations $\omega_{0}=0$, $\omega_{1}=1$,
$\omega_{2}=\tau$, $\omega_{3}=1+\tau$ and $\Lambda_{\tau}=\mathbb{Z+Z}\tau$,
where $\tau \in \mathbb{H}=\{  \tau|\operatorname{Im}\tau>0\}  $.
Define $E_{\tau}:=\mathbb{C}/\Lambda_{\tau}$ to be the elliptic curve and $E_{\tau}[2]:=\{ \frac{\omega_{k}}{2}|k=0,1,2,3\}+\Lambda
_{\tau}$ to be the set consisting of the lattice points and $2$-torsion points
in $E_{\tau}$. 

Let $\wp(z)=\wp(z;\tau)$ be the
Weierstrass elliptic function with periods $\Lambda_{\tau}$ and define $e_k(\tau):=\wp(\frac{\omega_k}{2};\tau)$, $k=1,2,3$. It is well known that
\[\wp'(z;\tau)^2=4\prod_{k=1}^3(\wp(z;\tau)-e_k(\tau))=4\wp(z;\tau)^3-g_2(\tau)\wp(z;\tau)-g_3(\tau),\]
where $g_2(\tau), g_3(\tau)$ are invariants of the elliptic curve $E_{\tau}$.
Let $\zeta(z)=\zeta(z;\tau):=-\int^{z}\wp(\xi;\tau)d\xi$
be the Weierstrass zeta function with two quasi-periods $\eta_{k}(\tau)$:
\begin{equation}
\eta_{k}(\tau):=2\zeta(\tfrac{\omega_{k}}{2} ;\tau)=\zeta(z+\omega_{k} ;\tau)-\zeta(z;\tau),\quad k=1,2,
\label{40-2}
\end{equation}
and $\sigma(z)=\sigma(z;\tau):=\exp \int^{z}\zeta(\xi;\tau)d\xi$ be the Weierstrass sigma function. Notice that $\zeta(z)$ is an odd
meromorphic function with simple poles at $\Lambda_{\tau}$ and $\sigma(z)$
is an odd holomorphic function with simple zeros at $\Lambda_{\tau}$.

This is the final one in our project of studying the generalized Lam\'{e} equation (denoted by H$(\mathbf{n},B,\tau)$)
\begin{equation}  \label{eq21}
y^{\prime \prime }(z)=(I_{\mathbf{n}}(z;\tau)+B)y(z),\quad z\in\mathbb{C},
\end{equation}
with the \emph{Darboux-Treibich-Verdier potential} \cite{GW1,TV0,TV1,TV,Veselov}
\begin{equation}
I_{\mathbf{n}}(z;\tau):=\sum_{k=0}^3n_k(n_k+1)\wp(z+\tfrac{\omega_k}{2};\tau),
\end{equation}
where $\mathbf{n}=(n_0, n_1, n_2, n_3)$ with $n_k\in\mathbb{N}$ and $\max n_k\geq 1$, and $B\in\mathbb C$ is a parameter.  By changing variable $z\to z+\frac{\omega_k}{2}$ if necessary, we can always assume $n_0\geq 1$.

In the 19th century, Darboux introduced H$(\mathbf{n},B,\tau)$ as the elliptic form of the well-known Heun equation (i.e. a second order Fuchsian differential equation with four regular singular points). About 100 years later, H$(\mathbf{n},B,\tau)$ was introduced in the soliton theory by Treibich and Verdier \cite{TV0,TV1,TV}. In a
series of papers \cite{Takemura1,Takemura2,Takemura3,Takemura4,Takemura5} by Takemura,
H$(\mathbf{n},B,\tau)$ was also studied as the eigenvalue problem for the
Hamiltonian of the $BC_{1}$ (one particle) Inozemtsev model \cite{Inozemtsev}.
When $\mathbf{n}=(n,0,0,0)$, H$(\mathbf{n},B,\tau)$ becomes the classical Lam\'{e} equation
\begin{equation}  \label{Lame}
y^{\prime \prime }(z)=[n(n+1)\wp(z;\tau)+B]y(z),\quad z\in\mathbb{C}.
\end{equation}
See the classic texts \cite{Halphen,Poole,Whittaker-Watson} and recent works \cite{CLW,Dahmen,LW2,Maier} for introductions about (\ref{Lame}).

Since the works of Treibich and Verdier \cite{TV0,TV1,TV}, the DTV potential $I_{\mathbf{n}}(z;\tau)$ is famous
as an algebro-geometric finite-gap potential associated with the stationary
KdV hierarchy. In the literature, a potential $q(z)$ is called an \emph{algebro-geometric finite-gap
potential }if there is an odd-order differential operator
\begin{equation}\label{odd-op}
P_{2g+1}=\left( \frac{d}{dz} \right)^{2g+1}%
+\sum_{j=0}^{2g-1}b_{j}(z)\left( \frac{d}{dz} \right)^{2g-1-j}
\end{equation} such that $[P_{2g+1}, d^{2}/dz^{2}
-q(z)]=0$, or equivalently, $q(z)$ is a solution of stationary KdV hierarchy equations (cf.
\cite{GH-Book,GW}).

For the DTV potential $I_{\mathbf{n}}(z;\tau)$, we let $P_{2g+1}$ be the unique operator of the form (\ref{odd-op}) satisfying $[P_{2g+1}, d^{2}/dz^{2}
-I_{\mathbf{n}}(z;\tau)]=0$ such that its order $2g+1$ is \emph{smallest}. Then
a theorem of Burchnall and Chaundy \cite{Burchnall-Chaundy} implies the existence of the so-called {\it spectral polynomial} $Q_{\mathbf{n}}(B;\tau)$ of degree $2g+1$ in $B$ associated to $I_{\mathbf{n}}(z;\tau)$ such that
\[P_{2g+1}^2=Q_{\mathbf{n}}(\tfrac{d^{2}}{dz^{2}}
-I_{\mathbf{n}}(z;\tau);\tau).\]
The number
$g$, i.e. the arithmetic genus of the associate hyperelliptic curve $\{(B,C)|C^{2}=Q_{\mathbf{n}}(B;\tau)\}$, was
computed by Gesztesy and Weikard \cite{GW1} (see also \cite{Takemura5}): Let $m_{k}$ be the rearrangement of $n_{k}$ such
that $m_{0}\geq m_{1}\geq m_{2}\geq m_{3}$, then
\begin{equation}
g=%
\begin{cases}
m_{0} & \text{if $\sum m_{k}$ is even and $m_{0}+m_{3}\geq m_{1}+m_{2}$};\\
\frac{m_{0}+m_{1}+m_{2}-m_{3}}{2} & \text{if $\sum m_{k}$ is even and
$m_{0}+m_{3}<m_{1}+m_{2}$};\\
m_{0} & \text{if $\sum m_{k}$ is odd and $m_{0}>m_{1}+m_{2}+m_{3}$};\\
\frac{m_{0}+m_{1}+m_{2}+m_{3}+1}{2} & \text{if $\sum m_{k}$ is odd and
$m_{0}\leq m_{1}+m_{2}+m_{3}$}.
\end{cases}
\label{genus}%
\end{equation}

In this paper, we continue our study, initiated in \cite{CKL1}, on H$(\mathbf{n}, B, \tau)$ from the monodromy aspect. Since the local exponents of H$(\mathbf{n}, B, \tau)$ at $\frac{\omega_{k}}{2}$ are $-n_{k}$, $n_{k}+1$ and $I_{\mathbf{n}}(z; \tau)$ is even elliptic, it is easily seen (cf. \cite{GW1,Takemura1}) that any solution is meromorphic in $\mathbb{C}$, i.e. the local monodromy matrix
at $\frac{\omega_{k}}{2}$ is the identity matrix $I_{2}$. Thus the monodromy representation $\rho:\pi_{1}(E_{\tau})  \to
SL(2,\mathbb{C})$ is a group homeomorphism, which is abelian and hence reducible. More precisely, let $\ell_{j}$, $j=1,2$, be two
fundamental cycles $z\to z+\omega_j$ of $E_{\tau}$, and let $(y_1(z),y_2(z))$ be any basis of solutions of H$(\mathbf{n}, B, \tau)$. Then there are monodromy matrices $\rho(\ell_j)\in SL(2,\mathbb{C})$ such that
\[\begin{pmatrix}y_1(z+\omega_j)\\ y_2(z+\omega_j)\end{pmatrix}=\rho(\ell_j)\begin{pmatrix}y_1(z)\\ y_2(z)\end{pmatrix},\quad j=1,2,\]
\[\rho(\ell_1)\rho(\ell_2)=\rho(\ell_2)\rho(\ell_1),\]
and the monodromy group is generated by $\rho(\ell_1), \rho(\ell_2)$.
Consequently, there are two cases (see \cite{CKL1,GW1}):

\begin{itemize}
\item[(a)] If $Q_{\mathbf{n}}(B;\tau)\neq 0$, the monodromy is completely reducible, namely up to a common conjugation,
$\rho(\ell_{1})$ and $\rho(\ell_{2})$ can be diagonized simultaneously and  expressed as%
\begin{equation}
\rho(\ell_{1})=%
\begin{pmatrix}
e^{-2\pi is} & 0\\
0 & e^{2\pi is}%
\end{pmatrix}
,\text{ \  \  \ }\rho(\ell_{2})=%
\begin{pmatrix}
e^{2\pi ir} & 0\\
0 & e^{-2\pi ir}%
\end{pmatrix}
\label{Mono-1}%
\end{equation}
for some $(r,s)\in \mathbb{C}^{2}\backslash \frac{1}{2}\mathbb{Z}^{2}$. In
particular,
\begin{equation}
(\text{tr}\rho(\ell_{1}),\text{tr}\rho(\ell_{2}))=(2\cos2\pi s,2\cos2\pi
r)\not \in \{ \pm(2,2),\pm(2,-2)\}. \label{complete-rs}%
\end{equation}
Define an equivalent relation\[(r,s)\sim (r',s')\quad\text{ if}\quad (r,s)\equiv \pm (r',s')\quad\operatorname{mod} \mathbb{Z}^2. \]
Since $\text{tr}\rho(\ell_{j})$ is independent of the choice of solutions, so $(r,s)$ is unique in $(\mathbb{C}^{2}\backslash \frac{1}{2}\mathbb{Z}^{2})/\sim$, and we refer it as the {\it monodromy data} of H$(\mathbf{n},B,\tau)$.

\item[(b)] If $Q_{\mathbf{n}}(B;\tau)=0$, then the monodromy is not completely reducible (i.e. the space of common eigenfunctions
is of dimension $1$), and up to a common conjugation, $\rho(\ell_{1})$ and
$\rho(\ell_{2})$ can be expressed as%
\begin{equation}
\rho(\ell_{1})=\varepsilon_{1}%
\begin{pmatrix}
1 & 0\\
1 & 1
\end{pmatrix}
,\text{ \  \  \ }\rho(\ell_{2})=\varepsilon_{2}%
\begin{pmatrix}
1 & 0\\
\mathcal{C} & 1
\end{pmatrix}
, \label{Mono-21}%
\end{equation}
where $\varepsilon_{1},\varepsilon_{2}\in \{ \pm1\}$ and $\mathcal{C}%
\in \mathbb{C}\cup \{ \infty \}$. In particular,
\begin{equation}
(\text{tr}\rho(\ell_{1}),\text{tr}\rho(\ell_{2}))=(2\varepsilon_{1}%
,2\varepsilon_{2})\in \{ \pm(2,2),\pm(2,-2)\}. \label{notcompleteC}%
\end{equation}
Remark that if $\mathcal{C}=\infty$, then (\ref{Mono-21}) should be understood
as%
\begin{equation*}
\rho(\ell_{1})=\varepsilon_{1}%
\begin{pmatrix}
1 & 0\\
0 & 1
\end{pmatrix}
,\text{ \  \  \ }\rho(\ell_{2})=\varepsilon_{2}%
\begin{pmatrix}
1 & 0\\
1 & 1
\end{pmatrix}
.
\end{equation*}

\end{itemize}

In view of Case (a), a natural question that interests us is \emph{how to characterize the mondoromy data $(r,s)$ in terms of $(\tau, B)$}. Define
\begin{equation}
\Sigma_{\mathbf{n}}:=\{(\tau, B)\in \mathbb{H}\times \mathbb{C}\,|\, Q_{\mathbf{n}}(B;\tau)\neq 0\},
\end{equation}
which is clearly an open connected subset of $\mathbb{H}\times \mathbb{C}$.
Then the map
\begin{equation}
\varphi_{\mathbf{n}}: \Sigma_{\mathbf{n}}\to (\mathbb{C}^2\setminus\tfrac12\mathbb{Z}^2)/\sim,\quad
\varphi_{\mathbf{n}}(\tau, B):=(r,s)
\end{equation}
is well-defined. It was proved in \cite[Lemma 2.3]{CKL2} that
\begin{equation}\label{fc-glo-uni}
\varphi_{\mathbf{n}}(\tau, B_1)\neq \varphi_{\mathbf{n}}(\tau, B_2)\quad \text{if }\quad B_1\neq B_2.
\end{equation}

\begin{remark}
Given any $(\tau_0, B_0)\in \Sigma_{\mathbf n}$, take $(r_0,s_0)\in \mathbb{C}^{2}\backslash \frac{1}{2}\mathbb{Z}^{2}$ to be a representative of  the monodromy data of H$(\mathbf{n}, B_0, \tau_0)$. Since there is a small neighborhood $V\subset \mathbb{C}^{2}\backslash \frac{1}{2}\mathbb{Z}^{2}$ of $(r_0,s_0)$ such that $(r,s)\not\sim (r',s')$ for any $(r,s), (r',s')\in V$ satisfying $(r,s)\neq (r',s')$, there is a small neighborhood $U\subset \Sigma_{\mathbf n}$ of $(\tau_0, B_0)$ such that $\varphi_{\mathbf n}\big|_{U}: U\to (\mathbb{C}^2\setminus\tfrac12\mathbb{Z}^2)/\sim$ can be seen as 
\[\varphi_{\mathbf n}\big|_{U}: U\to V\subset \mathbb{C}^2\setminus\tfrac12\mathbb{Z}^2,\]
and so we can consider the local analytic properties of $\varphi_{\mathbf n}$. Our main result of this paper is the following surprising univeral law.
\end{remark}

\begin{theorem}\label{thm-rsbtau} Given $\mathbf{n}$. Then the map $\varphi_{\mathbf{n}}$ is holomorphic and locally one-to-one, and satisfies
\begin{equation}\label{brstau}d\tau\wedge dB\equiv 8\pi^2 dr\wedge ds,\quad\forall (\tau, B)\in \Sigma_{\mathbf n}.\end{equation}
\end{theorem}

\begin{remark} This universal law \eqref{brstau} is quite mysterious to us. Is there any geometric explanation of this universal law? This question is worthy to be explored. 

On the other hand, as in \cite{GH-Book,GUW1}, for simplicity we call an elliptic function $q(z)$  \emph{an elliptic KdV potential} if $q(z)$ is a solution of the stationary KdV hierarchy.
The DTV potential $I_{\mathbf{n}}(z;\tau)$ is the simplest elliptic KdV potential. A natural question is \emph{whether any analogue of Theorem \ref{thm-rsbtau} holds for other elliptic KdV potentials}. More precisely, it was proved by Gesztesy, Unterkofler and Weikard \cite[Theorem 1.1]{GUW1} that $q(z)$ is an elliptic KdV potential if and only if up to adding a constant, $q(z)$ is expressed as
\begin{equation}
q(z)=q(z;\tau)=\sum_{j=1}^{n}m_{j}(m_{j}+1)\wp(z-p_{j}(\tau);\tau), \label{i2}%
\end{equation}
where $m_{j}\in\mathbb{N}$, $p_j(\tau)\in E_{\tau}$ satisfies $p_i(\tau)\neq p_j(\tau)$ for $i\neq j$ and the following conditions
\begin{equation}
\sum_{j=1,\neq i}^{n}m_{j}(m_{j}+1)\wp^{(2k-1)}(p_{i}(\tau)-p_{j}(\tau);\tau)=0\; \text{for
$1\leq k\leq m_{i}$, $1\leq i\leq n$}. \label{i3}%
\end{equation}
By \eqref{i3}, it is reasonable that $p_j(\tau)$'s are holomophic in $\tau$ and $p_i(\tau)\neq p_j(\tau)$ for $i\neq j$ for $\tau$ belonging to some open subset $O\subset \mathbb{H}$. Consider the corresponding differential equation
\begin{equation}y''(z)=\Big[\sum_{j=1}^{n}m_{j}(m_{j}+1)\wp(z-p_{j}(\tau);\tau)+B\Big]y(z),\quad \tau\in O,\end{equation}
and denote its spectral polynomial by $Q_{\mathbf p}(B;\tau)$. Then like the DTV case, the map 
\[\{(\tau, B)\in O\times \mathbb{C}\,|\, Q_{\mathbf{p}}(B;\tau)\neq 0\}\ni (\tau, B)\mapsto (r,s) \]
is well-defined. Is there any analogue of the universal law \eqref{brstau} holding for this map? 
One can see that our approach does not work for the general elliptic KdV potentials, and this question remains open. 
\end{remark}

Our proof of this universal law is based on Painlev\'e VI equation and the so-called \emph{pre-modular form} $Z^{\mathbf{n}}_{r,s}(\tau)$ constructed in \cite{CKL2}
which characterizes the monodromy data $(r,s)$ in a precise way.

\begin{definition}
A function $f_{r,s}(\tau)$ on $\mathbb{H}$, which depends meromorphically on two parameters $(r,s) (\operatorname{mod} \mathbb{Z}^2)\in\mathbb{C}^2$, is called a pre-modular form if the following hold:
\begin{itemize}
\item[(1)] If $(r,s)\in\mathbb{C}^2\setminus\frac{1}{2}\mathbb{Z}^2$, then $f_{r,s}(\tau)\not\equiv 0,\infty$ and is meromorphic in $\tau$. Furthermore, it is holomorphic in $\tau$ if $(r,s)\in\mathbb{R}^2\setminus\frac{1}{2}\mathbb{Z}^2$.
\item[(2)] There is $k\in\mathbb{N}$ independent of $(r,s)$ such that if $(r,s)$ is any $m$-torsion point for some $m\geq 3$, then $f_{r,s}(\tau)$ is a modular form of weight $k$ with respect to $\Gamma(m)$.
\end{itemize}
\end{definition}
The main result of \cite{CKL2} is following

\begin{theorem} \cite{CKL2}
\label{thm-premodular}  There exists a pre-modular form $Z_{r,s}^{\mathbf{n}%
}(\tau)$ defined in $\tau\in\mathbb{H}$ for any pair $(r,s)\in\mathbb{C}%
^2\setminus \frac{1}{2}\mathbb{Z}^2$ such that the following hold.

\begin{enumerate}
\item[(a)] If $(r,s)=(\frac{k_1}{m},\frac{k_2}{m})$ with $m\in 2\mathbb{N}
_{\geq 2}$, $k_1,k_2\in\mathbb{Z}_{\geq 0}$ and $\gcd(k_1,k_2,m)=1$, then $
Z_{r,s}^{\mathbf{n}}(\tau)$ is a modular form of weight $
\sum_{k=0}^3n_k(n_k+1)/2$ with respect to  $$\Gamma(m):=\{\gamma
\in SL(2,\mathbb{Z})|\gamma\equiv I_2\mod m \}.$$

\item[(b)] Given $(r,s)\in\mathbb{C}^2\setminus\frac{1}{2}\mathbb{Z}^2$ and $\tau_0\in\mathbb{H}$ such that $r+s\tau_0\notin \Lambda_{\tau_0}$,  there is $B\in\mathbb{C}$ such that $(r,s)$ is the monodromy data of H$(\mathbf{n}, B, \tau_0)$, i.e. the monodromy matrices of H$(\mathbf{n}, B, \tau_0)$ are given by \eqref{Mono-1}
\[
\rho(\ell_{1})=%
\begin{pmatrix}
e^{-2\pi is} & 0\\
0 & e^{2\pi is}%
\end{pmatrix}
,\text{ \  \  \ }\rho(\ell_{2})=%
\begin{pmatrix}
e^{2\pi ir} & 0\\
0 & e^{-2\pi ir}%
\end{pmatrix}
\]
if and only if $Z_{r,s}^{\mathbf{n}}(\tau_0)=0$.
\end{enumerate}
\end{theorem}

Theorem \ref{thm-premodular} for the Lam\'{e} case $n_1=n_2=n_3=0$ was first proved in \cite{LW2}.
We emphasize that Theorem \ref{thm-premodular} has important applications to nonlinear PDEs; see \cite{CKL2,LW2} for details.

In this paper, we prove that the assumption $r+s\tau_0\notin \Lambda_{\tau_0}$ in Theorem Theorem \ref{thm-premodular}-(b) can be deleted, i.e.

\begin{theorem}\label{thm-premodular-1}
Given $(r,s)\in\mathbb{C}^2\setminus\frac{1}{2}\mathbb{Z}^2$ and $\tau_0\in\mathbb{H}$,  there is $B\in\mathbb{C}$ such that $(r,s)$ is the monodromy data of H$(\mathbf{n}, B, \tau_0)$ if and only
if $Z_{r,s}^{\mathbf{n}}(\tau_0)=0$. 
\end{theorem}

After Theorem \ref{thm-premodular-1}, a further question arises: {\it What are the explicit expressions of these pre-modular forms}? This question is very difficult because the weight $\frac{1}{2}\sum n_k(n_k+1)$ is large for general $\mathbf{n}$. Define
\begin{align}\label{z-rs}Z_{r,s}(\tau):=&\zeta(r+s\tau;\tau)-r\eta_1(\tau)-s\eta_2(\tau)\\
=&\zeta(r+s\tau;\tau)-(r+s\tau)\eta_1(\tau)+2\pi is,\nonumber\end{align}
where we used the Legendre relation $\tau\eta_1-\eta_2=2\pi i$.
It is known from \cite{Dahmen,LW2} that (write $Z=Z_{r,s}(\tau)$, $\wp=\wp(r+s\tau|\tau)$ and $\wp^{\prime
}=\wp^{\prime}(r+s\tau|\tau)$ for convenience):
\begin{equation}\label{zrs12}
Z_{r,s}^{(1,0,0,0)}=Z,\quad Z_{r,s}^{(2,0,0,0)}=Z^{3}-3\wp Z-\wp^{\prime},
\end{equation}
\begin{align*}
Z_{r,s}^{(3,0,0,0)}=  &  Z^{6}-15\wp Z^{4}-20\wp^{\prime}Z^{3}+\left(
\tfrac{27}{4}g_{2}-45\wp^{2}\right)  Z^{2}\\
&  -12\wp \wp^{\prime}Z-\tfrac{5}{4}(\wp^{\prime})^{2}.
\end{align*}
{\allowdisplaybreaks%
\begin{align*}
Z_{r,s}^{(4,0,0,0)}=  &  Z^{10}-45\wp Z^{8}-120\wp^{\prime}Z^{7}+(\tfrac
{399}{4}g_{2}-630\wp^{2})Z^{6}-504\wp \wp^{\prime}Z^{5}\\
&  -\tfrac{15}{4}(280\wp^{3}-49g_{2}\wp-115g_{3})Z^{4}+15(11g_{2}-24\wp
^{2})\wp^{\prime}Z^{3} \\
&  -\tfrac{9}{4}(140\wp^{4}-245g_{2}\wp^{2}+190g_{3}\wp+21g_{2}^{2}%
)Z^{2}\label{z-n-4}\\
&  -(40\wp^{3}-163g_{2}\wp+125g_{3})\wp^{\prime}Z+\tfrac{3}{4}(25g_{2}%
-3\wp^{2})(\wp^{\prime})^{2}.
\end{align*}
}%
The above formulas are all for the Lam\'{e} case.
For $n\geq 5$, the explicit expression of $Z_{r,s}^{(n,0,0,0)}(\tau)$ is not known so far. See \cite{CKLW,Dahmen,LW2} for applications of the above formulas.

Here are new examples of $Z_{r,s}^{\bf n}(\tau)$ for the DTV case:
\[Z_{r,s}^{(1,1,0,0)}=Z^2-\wp+e_1,\]
\[Z_{r,s}^{(1,0,1,0)}=Z^2-\wp+e_2,\quad Z_{r,s}^{(1,0,0,1)}=Z^2-\wp+e_3,\]
\[Z_{r,s}^{(2,1,0,0)}=Z^4+3(e_1-2\wp)Z^2-4\wp'Z-3(\wp^2+e_1\wp+e_1^2-\tfrac{g_2}{4}),\]
and similarly, the expression of $Z_{r,s}^{(2,0,1,0)}$ (resp. $Z_{r,s}^{(2,0,0,1)}$) is obtained by replacing $e_1$ in $Z_{r,s}^{(2,1,0,0)}$ with $e_2$ (resp. $e_3$). See \cite{CKL2}.

In order to prove Dahmen and Beukers' conjectural formula of counting integral Lam\'{e} equations with finite monodromy, we proved in \cite{CKL-Dahmen} that $Z_{r,s}^{(n,0,0,0)}(\tau)$ has at most simple zeros. This result is not trivial at all due to two reasons: (1) The explicit expression of $Z_{r,s}^{(n,0,0,0)}(\tau)$ is not known for $n\geq 5$; (2) Even for $n=2,3,4$, the expressions of $Z_{r,s}^{(n,0,0,0)}(\tau)$ are already so complicated that we can not obtain this simple zero property by calculating the highly nontrivial derivative $\frac{d}{d\tau}Z_{r,s}^{(n,0,0,0)}(\tau)$. In \cite{CKL-Dahmen} we proved this simple zero property by showing that $Z_{r,s}^{(n,0,0,0)}(\tau)$ appears in the denominator of expressions of solutions to
certain Painlev\'{e} VI equation. We believe that this assertion should also holds for $Z_{r,s}^{\bf n}(\tau)$, i.e. $Z_{r,s}^{\bf n}(\tau)$ should also appear in the denominator of expressions of solutions to
certain Painlev\'{e} VI equation. But this assertion has not been confirmed so far.

In this paper, we develop a new idea to extend the simple zero property to include the DTV potential. 

\begin{theorem}\label{thm-simplezero} Given $\mathbf{n}$. Then for any fixed $(r,s)\in\mathbb{C}^2\setminus\frac{1}{2}\mathbb{Z}^2$, $Z_{r,s}^{\mathbf{n}
}(\tau)$ has at most simple zeros in $\mathbb{H}$.
\end{theorem}

The rest of this paper is organized as follows. In Section 2, we briefly review the construction of the pre-modular form $Z_{r,s}^{(\mathbf{n})}(\tau)$ from \cite{CKL1,CKL2} and prove Theorem \ref{thm-premodular-1}. In Section 3, we establish the connection between Painlev\'{e} VI equation and  the pre-modular form $Z_{r,s}^{(\mathbf{n})}(\tau)$. In Section 4 we prove Theorem \ref{thm-simplezero}. In Section 5, we prove the universal law for two simplest Lam\'{e} case. In Section 6, we prove the universal law for general DTV cases via an induction approach, where Painlev\'{e} VI equation plays a crucial role. Finally in Section 7, we apply the universal law to the algebraic multiplicity of (anti)periodic eigenvalues of the associated Hill operator.
\section{Pre-modular forms}

In this section we prove Theorem \ref{thm-premodular-1}. For this purpose we need to briefly review the constuction of the pre-modular form $Z_{r,s}^{(\mathbf{n})}(\tau)$ from \cite{CKL1,CKL2}. Denote
\[|\mathbf{n}|:=\sum_k n_k\quad\text{for}\;\;\mathbf{n}.\]

(i) Any solution of $H({\bf n},B,\tau)$ is meromorphic in $\mathbb{C}$. Given any $B\in\mathbb{C}$, there exist a unique pair $\pm \boldsymbol{a}:=\pm\{a_1,\cdots,a_{|\mathbf{n}|}\}\subset E_{\tau}$ and constants $c(\pm\boldsymbol{a})\in\mathbb{C}$ (see (\ref{ca-ex}) below) such that
\begin{equation}\label{yby}y_{\boldsymbol{a}}(z)
:=e^{c({\boldsymbol{a}})z}\frac{\prod_{i=1}^{|\mathbf{n}|}\sigma(z-a_i)}{\prod_{k=0}^3\sigma(z-\frac{\omega_k}{2})^{n_k}},\; y_{-\boldsymbol{a}}(z)
:=e^{c({-\boldsymbol{a}})z}\frac{\prod_{i=1}^{|\mathbf{n}|}\sigma(z+a_i)}{\prod_{k=0}^3\sigma(z-\frac{\omega_k}{2})^{n_k}}\end{equation}
are solutions of $H({\bf n},B,\tau)$. Since the DTV potential $I_{\mathbf n}(z;\tau)$ is an even function, $y_{\boldsymbol{a}}(-z)$ is also a solution of $H({\bf n},B,\tau)$ and has the same zeros as $y_{-\boldsymbol{a}}(z)$, so $y_{\boldsymbol{a}}(-z)$ and $y_{-\boldsymbol{a}}(z)$ are linearly dependent.
From here and the transformation law (let $\eta_3=2\zeta(\frac{\omega_3}{2})=\eta_1+\eta_2$)
\begin{equation}\label{tran-law}\sigma(z+\omega_k)=-e^{\eta_k(z+\frac{\omega_k}{2})}\sigma(z),\quad k=1,2,3,\end{equation} it is easy to see that
$y_{-\boldsymbol{a}}(z)=(-1)^{n_1+n_2+n_3}y_{\boldsymbol{a}}(-z)$ and $c(-\boldsymbol{a})=-c(\boldsymbol{a})-\sum_{k=1}^3n_k\eta_k$.

(ii) Recalling the spectral polynomial $Q_{\bf n}(B;\tau)$, $y_{\boldsymbol{a}}(a)$ and $y_{-\boldsymbol{a}}(z)$ are linearly independent if and only if $Q_{\bf n}(B;\tau)\neq 0$, which is also equivalent to
\begin{equation}\label{a-0a}\{a_1,\cdots,a_{|\mathbf{n}|}\}\cap \{-a_1,\cdots,-a_{|\mathbf{n}|}\}=\emptyset\quad\text{in }\;E_{\tau}.
\end{equation}
In this case,
\begin{equation}
c(\boldsymbol{a})=\sum_{i=1}^{|\mathbf{n}|}\zeta(a_{i})  -\sum_{k=1}^{3}\frac{n_{k}\eta_{k}
}{2},\label{ca-ex}
\end{equation}
and the $(r,s)$ defined by
\begin{equation}
\left \{
\begin{array}
[c]{l}%
\sum_{i=1}^{|\mathbf{n}|}a_{i}-\sum_{k=1}^{3}\frac{n_{k}\omega_{k}}{2}=r+s\tau \\
\sum_{i=1}^{|\mathbf{n}|}\zeta(a_{i})  -\sum_{k=1}^{3}\frac{n_{k}\eta_{k}
}{2}=r\eta_{1}+s\eta_{2}
\end{array}
\right. \label{rs2}%
\end{equation}
satisfies $(r,s)\in\mathbb{C}^2\setminus \frac{1}{2}\mathbb{Z}^2$, and
\[\begin{pmatrix}y_{\boldsymbol{a}}(z+\omega_1)\\ y_{-\boldsymbol{a}}(z+\omega_1)\end{pmatrix}=\begin{pmatrix}
e^{-2\pi is} & 0\\
0 & e^{2\pi is}%
\end{pmatrix}\begin{pmatrix}y_{\boldsymbol{a}}(z)\\ y_{-\boldsymbol{a}}(z)\end{pmatrix},\]
\[\begin{pmatrix}y_{\boldsymbol{a}}(z+\omega_2)\\ y_{-\boldsymbol{a}}(z+\omega_2)\end{pmatrix}=\begin{pmatrix}
e^{2\pi ir} & 0\\
0 & e^{-2\pi ir}%
\end{pmatrix}\begin{pmatrix}y_{\boldsymbol{a}}(z)\\ y_{-\boldsymbol{a}}(z)\end{pmatrix},\]
i.e. with respect to $y_{\boldsymbol{a}}(z)$ and $y_{-\boldsymbol{a}}(z)$, the monodromy matrices are given by
\begin{equation}
\rho(\ell_{1})=%
\begin{pmatrix}
e^{-2\pi is} & 0\\
0 & e^{2\pi is}%
\end{pmatrix}
,\text{ }\rho(\ell_{2})=%
\begin{pmatrix}
e^{2\pi ir} & 0\\
0 & e^{-2\pi ir}%
\end{pmatrix}.
\label{61.3512}%
\end{equation}

(iii) Define
\begin{equation}
Y_{\mathbf{n}}(\tau)  :=\left \{
\begin{array}
[c]{r}%
\boldsymbol{a}=\{a_{1},\cdot \cdot \cdot,a_{|\mathbf{n}|}\}\text{ }\in $Sym$^{|\mathbf{n}|}E_{\tau}\,\big|\,\text{
}y_{\boldsymbol{a}}(z)  \text{ defined in }\\
\text{(\ref{yby}) is a solution of H}(\mathbf{n},B,\tau) \text{ for some } B
\end{array}
\right \}. \label{set}%
\end{equation}
Then $\overline{Y_{\mathbf{n}}(\tau)}=Y_{\mathbf{n}}(\tau)\cup \{\infty_0\}$ is a hyperelliptic curve with arithmetic genus $g$ with
\[
\infty_0:=\bigg(  \overset{n_{0}}{\overbrace{0,\cdot \cdot \cdot,0}}%
,\overset{n_{1}}{\overbrace{\tfrac{\omega_{1}}{2},\cdot \cdot \cdot,\tfrac
{\omega_{1}}{2}}},\overset{n_{2}}{\overbrace{\tfrac{\omega_{2}}{2},\cdot
\cdot \cdot,\tfrac{\omega_{2}}{2}}},\overset{n_{3}}{\overbrace{\tfrac{\omega_{3}%
}{2},\cdot \cdot \cdot,\tfrac{\omega_{3}}{2}}}\bigg)  .
\]
The affine part
\begin{equation}\label{eq-affine}Y_{\mathbf{n}}(\tau)\simeq \{(B,C)|C^{2}%
=Q_{\mathbf{n}}(B;\tau)\},\end{equation}
and the
branch points of $Y_{\mathbf{n}}(\tau)$ are precisely those $\{a_{i}\}_{i=1}^{|\mathbf{n}|}\in
Y_{n}(\tau)$ such that
\begin{equation}\label{bran}\{a_1,\cdots,a_{|\mathbf{n}|}\}
=\{-a_1,\cdots,-a_{|\mathbf{n}|}\}\quad\text{in }\;E_{\tau}.\end{equation}

(iv)
The first formula of (\ref{rs2}) motivates us to study the addition map $\sigma
_{\mathbf{n}}:\overline{Y_{\mathbf{n}}(  \tau )}\rightarrow E_{\tau}$ (also called a covering map in \cite[Section 4]{Takemura4}):
\begin{equation}\label{fc-fc3}
\sigma_{\mathbf{n}}( \boldsymbol{a}) : =\sum_{i=1}^{N}
a_{i}  -\sum_{k=1}^{3}\tfrac{n_{k}\omega_{k}}{2}.%
\end{equation}
Since $2\sum_{k=1}^{3}\tfrac{n_{k}\omega_{k}}{2}=0$ in $E_{\tau}$, we have
\[
\sigma_{\mathbf{n}}( -\boldsymbol{a})  =-\sum_{i=1}^{N}
a_{i}  -\sum_{k=1}^{3}\tfrac{n_{k}\omega_{k}}{2}=-\sigma_{\mathbf{n}}( \boldsymbol{a})\quad\text{in }\;E_{\tau}.%
\]
Since the algebraic curve $\overline{Y_{\mathbf{n}}(\tau)}$ is irreducible, $\sigma_{\mathbf{n}}$ is a {\it finite morphism} and $\deg\sigma_{\mathbf{n}}$ is well-defined. We proved in \cite{CKL1} that
\[
\deg\sigma_{\mathbf{n}}=\frac{1}{2}\sum_{k=0}^3 n_k(n_k+1).
\]
Let $K(\overline{Y_{\mathbf{n}}(\tau)})$ be the field of rational functions on $\overline
{Y_{\mathbf{n}}(\tau)}$. Then
$K(\overline{Y_{\mathbf{n}}(\tau)})$ is a finite extension of $K(E_{\tau})$ and%
\[
\left[  K(\overline{Y_{\mathbf{n}}(\tau)}):K(E_{\tau})\right]  =\frac{1}{2}\sum_{k=0}^3 n_k(n_k+1).
\]

(v) Define $\mathbf{z}_{\mathbf{n}}: \overline{Y_{\mathbf{n}}(\tau)}\to\mathbb{C}\cup\{\infty\}$ by
\[\mathbf{z}_{\mathbf{n}}(a_1,\cdots,a_N):=\zeta\Bigg(\sum_{i=1}^N a_i-\sum_{k=1}^{3}\frac{n_{k}\omega_{k}}{2}\Bigg)-\sum_{i=1}^{N}\zeta(a_{i})  +\sum_{k=1}^{3}\frac{n_{k}\eta_{k}
}{2}.\]
Then $\mathbf{z}_{\mathbf{n}}\in K(\overline{Y_{\mathbf{n}}(\tau)})$ is a primitive generator of the finite field
extension of $K(\overline{Y_{\mathbf{n}}(\tau)})$ over $K(E_{\tau})$, and there is a minimal polynomial
\begin{equation}
W_{\mathbf{n}}(X)=W_{\mathbf{n}}(X;\sigma_{\mathbf{n}},\tau)\in \mathbb{Q}[e_1(\tau),e_2(\tau
),e_3(\tau),\wp(\sigma_{\mathbf{n}};\tau),\wp^{\prime}(\sigma_{\mathbf{n}};\tau)][X] \label{kk-ll}
\end{equation}
of the field extension $K(\overline{Y_{\mathbf{n}}(\tau)})$ over $K(E_{\tau})$ which defines the
covering map $\sigma_{\mathbf{n}}$ between algebraic curves.
Then the pre-modular form $Z_{r,s}^{\mathbf{n}}(\tau)$ is construct from $W_{\mathbf{n}}(X)$ by the following result.

\begin{theorem} \cite[Theorem 2.4]{CKL2} \label{thm-5A}

\begin{itemize}
\item[(1)] $W_{\mathbf{n}}(X;\sigma_{\mathbf{n}},\tau)$ is a monic polynomial of
$X$-degree $\frac{1}{2}\sum_k n_k(n_k+1)$ such that
\[
W_{\mathbf{n}}(\mathbf{z}_{\mathbf {n}}(\boldsymbol{a});\sigma_{\mathbf {n}}(\boldsymbol{a}),\tau)=0.
\]
Moreover, $W_{\mathbf{n}}$ is homogenous of weight $\frac{1}{2}\sum_k n_k(n_k+1)$, where the weights of $X$, $\wp(\sigma_{\mathbf{n}})$, $e_k$'s, $\wp'(\sigma_{\mathbf{n}})$ are $1$, $2$, $2$, $3$ respectively.

\item[(2)] Fix any $\tau$. For each $\sigma\in E_{\tau
}\backslash E_{\tau}[2]$ being outside the branch loci of $\sigma
_{\mathbf{n}}:\overline{Y_{\mathbf{n}}(\tau)}\rightarrow E_{\tau}$, $W_{\mathbf{n}}(\cdot;\sigma,\tau)$
has $\frac{1}{2}\sum_k n_k(n_k+1)$ distinct zeros.

\item[(3)] Define $Z_{r,s}
^{\mathbf{n}}(\tau):=W_{\mathbf{n}}(Z_{r,s}(\tau);r+s\tau,\tau)$. Then $Z_{r,s}
^{\mathbf{n}}(\tau)$ is the desired pre-modular form as stated in Theorem \ref{thm-premodular}.
\end{itemize}
\end{theorem}

Here we have the following simple observation.

\begin{lemma}\label{lem-s-finite}
Fix any $\tau_0$. There is a polynomial $g_{\mathbf{n}}(s)$ of degree $\frac{1}{2}\sum{n_k(n_k+1)}$ such that if $(r_0,s_0)\in\mathbb{C}^2\setminus\frac{1}{2}\mathbb{Z}^2$ satisfies $r_0+s_0\tau_0=0$ and $Z_{r_0,s_0}^{\mathbf{n}}(\tau_0)=0$, then $g_{\mathbf{n}}(s_0)=0$.
\end{lemma}

\begin{proof}
Clearly $Z_{r_0,s}^{\mathbf{n}}(\tau_0)$ is meromorphic in $s$. Denote $\alpha=r_0+s\tau_0$, i.e. $\alpha\to 0$ as $s\to s_0$. It is well known that
\[\zeta(\alpha;\tau_0)=\frac{1}{\alpha}+\sum_{j=1}^\infty a_j\alpha^{2j+1},\]
\[\wp(\alpha;\tau_0)=\frac{1}{\alpha^2}-\sum_{j=1}^\infty (2j+1)a_j\alpha^{2j},\]
\[\wp'(\alpha;\tau_0)=\frac{-2}{\alpha^3}-\sum_{j=1}^\infty 2j(2j+1)a_j\alpha^{2j-1},\]
where $a_j\in \mathbb{Q}[g_2(\tau_0),g_3(\tau_0)]\subset \mathbb{Q}[e_1(\tau_0),
e_2(\tau_0),e_3(\tau_0)]$.
Then \eqref{z-rs} gives
\[Z_{r_0,s}(\tau_0)=\frac{1}{\alpha}+2\pi i s-\eta_1(\tau_0)\alpha+\sum_{j=1}^\infty a_j\alpha^{2j+1}.\]
Recalling (\ref{kk-ll}) and Theorem \ref{thm-5A} that
\begin{align}\label{eq-eq}
Z_{r_0,s}^{\mathbf{n}}(\tau_0)=Z_{r_0,s}(\tau_0)^{\frac{1}{2}\sum_k n_k(n_k+1)}+\sum_{j=0}^{\frac{1}{2}\sum_k n_k(n_k+1)-2}b_jZ_{r_0,s}(\tau_0)^{j},
\end{align}
where \[b_j\in\mathbb{Q}[\wp(\alpha;\tau_0),\wp'(\alpha;\tau_0),e_1(\tau_0),
e_2(\tau_0),e_3(\tau_0)]\]
is homogeneous weight of $\frac{1}{2}\sum_k n_k(n_k+1)-j$, where the weights of $\wp(\alpha;\tau_0)$, $e_k(\tau_0), \wp'(\alpha;\tau_0)$ are $2,2,3$ respectively.
Inserting the above expansions into (\ref{eq-eq}), we obtain
\[Z_{r_0,s}^{\mathbf{n}}(\tau_0)=\sum_{j=-\frac{1}{2}\sum_k n_k(n_k+1)}^{+\infty}d_j \alpha^{j},\]
where
\[d_j=d_j(s)\in \mathbb{Q}[\eta_1(\tau_0),e_1(\tau_0),
e_2(\tau_0),e_3(\tau_0)][2\pi i s],\]
in particular, \[d_{0}(s)=(2\pi i s)^{\frac{1}{2}\sum_k n_k(n_k+1)}+\text{lower order terms}\]
is of degree $\frac{1}{2}\sum_k n_k(n_k+1)$. Now by $Z_{r_0,s_0}^{\mathbf{n}}(\tau_0)=0$ and $\alpha\to 0$ as $s\to s_0$, we obtain
\[d_j(s_0)=0\quad\text{for all }\; j\leq 0.\]
The proof is complete by letting $g_{\mathbf{n}}(s)=d_{0}(s)$.
\end{proof}

Now we apply the above theory to prove Theorem \ref{thm-premodular-1}.

\begin{proof}[Proof of Theorem \ref{thm-premodular-1}] Fix any $(r_0,s_0)\in\mathbb{C}^2\setminus\frac{1}{2}\mathbb{Z}^2$ and $\tau_0\in\mathbb{H}$ such that $r_0+s_0\tau_0\in\Lambda_{\tau_0}$. By replacing $(r_0,s_0)$ with $(r_0+m_1,s_0+m_2)$, $m_1, m_2\in\mathbb{Z}$, if necessary, we may assume $r_0+s_0\tau_0=0$.

{\bf Step 1.} We prove the sufficient part.
Suppose there is $B_0$ such that $(r_0,s_0)$ is the monodromy data of H$(\mathbf{n},B_0,\tau_0)$, we need to prove $Z_{r_0,s_0}^{\mathbf{n}}(\tau_0)=0$.

By $Q_{\mathbf{n}}(B_0;\tau_0)\neq 0$, there is small $\varepsilon>0$ such that for any  $|B-B_0|<\varepsilon$, $Q_{\mathbf{n}}(B;\tau_0)\neq 0$. Consequently, we may assume that the monodromy data of H$(\mathbf{n},B,\tau_0)$ is $(r(B),s(B))\notin\frac12\mathbb{Z}^2$ satisfying $(r(B),s(B))\to (r_0,s_0)$ as $B\to B_0$. Since the addition map $\sigma_{\mathbf{n}}$ is a finite morphism, the pre-image $\sigma_{\mathbf{n}}^{-1}(0)$ is finite, so except finite $B$'s we have $r(B)+s(B)\tau_0\notin \Lambda_{\tau_0}$ and then Theorem \ref{thm-premodular} implies
$Z_{r(B),s(B)}^{\mathbf{n}}(\tau_0)=0.$
From here and the continuity of $Z_{r,s}^{\mathbf{n}}(\tau_0)$ with respect to $(r,s)$, we obtain $Z_{r_0,s_0}^{\mathbf{n}}(\tau_0)=0$.

{\bf Step 2.} We prove the necessary part. Suppose $Z_{r_0,s_0}^{\mathbf{n}}(\tau_0)=0$, we need to prove the existence of $B_0$ such that $(r_0,s_0)$ is the monodromy data of H$(\mathbf{n},B_0,\tau_0)$.

Since $Z_{r,s}^{\mathbf{n}}(\tau_0)$ is meromorphic in $(r,s)\notin\frac12\mathbb{Z}^2$, there is small $\varepsilon>0$ such that for any $|s-s_0|<\varepsilon$, $Z_{r,s}^{\mathbf{n}}(\tau_0)$ as a function of $r$ has a zero $r(s)$ satisfying $r(s)\to r_0$ as $s\to s_0$. Then by Lemma \ref{lem-s-finite} and by taking $\varepsilon$ smaller if necessarily, we have $(r(s),s)\notin\frac12\mathbb{Z}^2$ and $r(s)+s\tau_0\notin\Lambda_{\tau_0}$ for any $0<|s-s_0|<\varepsilon$. From here and $Z_{r(s),s}^{\mathbf{n}}(\tau_0)=0$, it follows from Theorem \ref{thm-premodular} that there is $B(s)$ such that $(r(s),s)$ is the monodromy data of H$(\mathbf{n},B(s),\tau_0)$.
Recalling (\ref{ca-ex})-(\ref{rs2}) that the corresponding $c(\boldsymbol{a})=r(s)\eta_1(\tau_0)+s\eta_2(\tau_0)$,
since we proved in \cite[(5.7)]{CKL1} that
$B(s)\to \infty$ if and only if the corresponding $c(\boldsymbol{a})\to \infty$, we conclude from $(r(s),s)\to (r_0,s_0)$ that $B(s)$ are uniformly bounded as $s\to s_0$ and so up to a subsequence, $B(s)\to B_0$ for some $B_0$. Consequently, $(r_0,s_0)$ is the monodromy data of H$(\mathbf{n},B_0,\tau_0)$. This completes the proof.
\end{proof}

\section{Painlev\'{e} VI equation and Pre-modular forms}

Our proofs of Theorems \ref{thm-rsbtau} and \ref{thm-simplezero} are based on the connection between H$(\mathbf{n}, B, \tau)$ and Painlev\'{e} VI equation. First we briefly review some basic facts about Painlev\'{e} VI equation.

\subsection{Painlev\'{e} VI equation}
The well-known Painlev\'{e} VI equation (PVI) with four free parameters
$(\alpha,\beta,\gamma,\delta)$  is written
as
{\allowdisplaybreaks
\begin{align}
\frac{d^{2}\lambda}{dt^{2}}=  &  \frac{1}{2}\left(  \frac{1}{\lambda}+\frac
{1}{\lambda-1}+\frac{1}{\lambda-t}\right)  \left(  \frac{d\lambda}{dt}\right)
^{2}-\left(  \frac{1}{t}+\frac{1}{t-1}+\frac{1}{\lambda-t}\right)
\frac{d\lambda}{dt}\nonumber \\
&  +\frac{\lambda(\lambda-1)(\lambda-t)}{t^{2}(t-1)^{2}}\left[  \alpha
+\beta \frac{t}{\lambda^{2}}+\gamma \frac{t-1}{(\lambda-1)^{2}}+\delta
\frac{t(t-1)}{(\lambda-t)^{2}}\right]  . \label{46}%
\end{align}
}%
Due to its connection with many different disciplines in mathematics and
physics, PVI has been extensively studied in the past several
decades. We refer the readers to the texts \cite{FIK,GP} for a detailed introduction of PVI.

One of the fundamental properties for PVI is the so-called
\emph{Painlev\'{e} property}, which says that any solution $\lambda(t)$ of
PVI has neither movable branch points nor movable essential
singularities; in other words, for any $t_{0}\in \mathbb{C}\backslash \{0,1\}$,
either $\lambda(t)$ is holomorphic at $t_{0}$ or $\lambda(t)$ has a pole at $t_{0}$. Therefore, it is reasonable to lift PVI to
the covering space $\mathbb{H=}\{ \tau|\operatorname{Im}\tau>0\}$ of
$\mathbb{C}\backslash \{0,1\}$ by the following transformation:%
\begin{equation}
t=\frac{e_{3}(\tau)-e_{1}(\tau)}{e_{2}(\tau)-e_{1}(\tau)},\text{ \ }%
\lambda(t)=\frac{\wp(p(\tau);\tau)-e_{1}(\tau)}{e_{2}(\tau)-e_{1}(\tau)}.
\label{II-130}%
\end{equation}
Then $\lambda(t)$ solves PVI if and only if $p(\tau)$ satisfies the following
\emph{elliptic form} of PVI (EPVI):
\begin{equation}
\frac{d^{2}p(\tau)}{d\tau^{2}}=\frac{-1}{4\pi^{2}}\sum_{k=0}^{3}\alpha_{k}%
\wp^{\prime}\left( p(\tau)+\tfrac{\omega_{k}}{2};
\tau \right)  , \label{124}%
\end{equation}
with parameters given by%
\begin{equation}
\left(  \alpha_{0},\alpha_{1},\alpha_{2},\alpha_{3}\right)  =\left(
\alpha,-\beta,\gamma,\tfrac{1}{2}-\delta \right)  . \label{126-0}%
\end{equation}
See e.g. \cite{Babich-Bordag,Y.Manin} for the proof. The Painlev\'{e} property implies that
function $\wp(p(\tau)|\tau)$ is a single-valued meromorphic function in
$\mathbb{H}$. This is an advantage of making the transformation (\ref{II-130}).

\begin{remark}
\label{identify}Clearly for any $m_{1},m_{2}\in \mathbb{Z}$, $\pm p(\tau
)+m_{1}+m_{2}\tau$ is also a solution of the elliptic form (\ref{124}). Since
they all give the same solution $\lambda(t)$ of PVI via (\ref{II-130}), we
always identify all these $\pm p(\tau)+m_{1}+m_{2}\tau$ with the
same one $p(\tau)$.
\end{remark}

From now on we consider PVI with parameters
\begin{align}
(\alpha,\beta,\gamma,\delta)=  &  \left(  \tfrac{1}{2}(n_{0}+\tfrac{1}{2}%
)^{2},\text{ }-\tfrac{1}{2}(n_{1}+\tfrac{1}{2})^{2},\text{ }\tfrac{1}{2}%
(n_{2}+\tfrac{1}{2})^{2},\right. \nonumber \\
&  \left.  \tfrac{1}{2}-\tfrac{1}{2}(n_{3}+\tfrac{1}{2})^{2}\right)  ,\text{
}n_{k}\in \mathbb{N}\text{ for all }k, \label{parameter}%
\end{align}
and denoted it by PVI$_{\mathbf{n}}$;
or equivalently
EPVI with parameters
\begin{equation}
\alpha_{k}=\tfrac{1}{2}(n_{k}+\tfrac{1}{2})^{2},\text{ }n_{k}\in
\mathbb{N}\text{ for all }k,
\label{parameter0}%
\end{equation}
and denoted it by EPVI$_{\mathbf{n}}$.

First we recall Hitchin's famous formula for the case $\mathbf{n}=\mathbf{0}$.
For any $(r,s)  \in \mathbb{C}^{2}\backslash \frac
{1}{2}\mathbb{Z}^{2}$, let $p_{r,s}^{\mathbf{0}}(\tau
)$ be defined by
\begin{equation}
\wp(p_{r,s}^{\mathbf{0}}(\tau);\tau):=\wp(r+s\tau;\tau)+\frac{\wp^{\prime
}(r+s\tau;\tau)}{2Z_{r,s}(\tau)  }. \label{513-1}%
\end{equation}
In \cite{Hit1} Hitchin proved the following remarkable result for EPVI$_{\mathbf{0}}$.

\begin{theorem} \cite{Hit1} \label{thm-Hitchin} For any $(
r,s)  \in \mathbb{C}^{2}\backslash \frac{1}{2}\mathbb{Z}^{2}$\textit{,
}$p_{r,s}^{\mathbf{0}}(\tau)$ given by (\ref{513-1}) is a solution to
EPVI$_{\mathbf{0}}$;
or equivalently, $\lambda_{r,s}^{\mathbf{0}}(t):=
\frac{\wp(p_{r,s}^{\mathbf{0}}(\tau);\tau)-e_{1}(  \tau)  }%
{e_{2}(\tau)-e_{1}(\tau)}$ via (\ref{513-1}) is a solution to PVI$_{\mathbf{0}}$.\end{theorem}

It is known (cf. \cite[Section 5]{Chen-Kuo-Lin}) that solutions of EPVI$_{\mathbf{n}}$ could be
obtained from solutions of EPVI$_{\mathbf{0}}$ via the well-known Okamoto transformations
(\cite{Okamoto1}).
\medskip

\noindent{\bf Notation:} \emph{Denote by $p_{r,s}^{\mathbf{n}}(\tau)$ to be the solution of EPVI$_{\mathbf{n}}$ obtained from $p_{r,s}^{\mathbf{0}}(\tau)$
in Theorem \ref{thm-Hitchin} via the Okamoto transformations}.
\medskip

Then by applying Hitchin's formula (\ref{513-1}) and the Okamoto transformation, we proved in \cite[Remark 5.2]{Chen-Kuo-Lin} that

\begin{lemma} \cite[Remark 5.2]{Chen-Kuo-Lin} \label{lem-2.3} Given $\mathbf{n}$,
there is a rational function $\Xi_{\mathbf{n}}(\cdot,\cdot,\cdot,\cdot,\cdot,\cdot)$ of six independent variables with coefficients in $\mathbb{Q}$ such that for any $(r,s)\in\mathbb{C}^2\setminus\frac12\mathbb{Z}^2$, there holds
\[
\wp(p_{r,s}^{\mathbf{n}}(\tau);\tau)= \Xi_{\mathbf{n}}(Z_{r,s}(\tau),\wp(r+s\tau;\tau),\wp'(r+s\tau;\tau),
e_1(\tau),e_2(\tau),e_3(\tau)).
\]
\end{lemma}

For example, by writing \[Z=Z_{r,s}(\tau), \quad\wp=\wp(r+s\tau;\tau), \quad\wp'=\wp'(r+s\tau;\tau)\] for convenience, we have
\begin{equation}
\wp(p_{r,s}^{(1,0,0,0)}(\tau);\tau)=\wp+\frac{3\wp^{\prime}Z^{2}+\left(
12\wp^{2}-g_{2}\right)  Z+3\wp\wp^{\prime}}{2(Z^{3}-3\wp
Z-\wp^{\prime})}. \label{expression-1}%
\end{equation}
Note from \eqref{zrs12} that $Z_{r,s}^{(2,0,0,0)}(\tau)=Z^{3}-3\wp
Z-\wp^{\prime}$ appears in the denominator of $\wp(p_{r,s}^{(1,0,0,0)}(\tau);\tau)$. We will study the general relation between $Z_{r,s}^{\mathbf{n}}(\tau)$ and PVI in Section 3.2.

On the other hand, we proved in \cite{Chen-Kuo-Lin0} that  EPVI$_{\mathbf{n}}$  governs the isomonodromic deformation of the following generalized
Lam\'{e} equation
(denoted it by GLE$(\mathbf{n},p,A,\tau)$)
\begin{align} \label{89-1}
y''=\Big[
&\sum_{k=0}^{3}n_{k}(n_{k}+1)\wp(z+\tfrac{\omega_{k}}{2};\tau)+\frac{3}{4}
(\wp(z+p;\tau)+\wp(z-p;\tau))\\
&+A(\zeta(z+p;\tau)-\zeta(z-p;\tau))+B
\Big]  y=:I(z)y,\nonumber
\end{align}
where $\pm p\not \in E_{\tau}[2]$ and 
\begin{equation}
B=A^{2}-\zeta(2p;\tau)A-\tfrac{3}{4}\wp(2p;\tau)-\sum_{k=0}^{3}n_{k}(n_{k}+1)\wp(
p+\tfrac{\omega_{k}}{2};\tau).  \label{101}%
\end{equation}
Note that \eqref{101} is equivalent to saying that $\pm p\not \in E_{\tau}[2]$ are always
\emph{apparent singularities} (i.e. any solution of \eqref{89-1} has no logarithmic singularities at $\pm p$).

Fix any base point $q_{0}\in E_{\tau}$ that is not a singularity of \eqref{89-1}. The monodromy representation of GLE
(\ref{89-1}) is a homomorphism $\rho:\pi_{1}(  E_{\tau}\backslash
(\{  \pm p\}  \cup E_{\tau}[2]),q_{0})  \rightarrow
SL(2,\mathbb{C})$. Since $n_k\in \mathbb{N}$ and the local exponents
of (\ref{89-1}) at $\frac{\omega_k}{2}$ are $-n_k$ and $n_k+1$, the
local monodromy matrix at $\frac{\omega_k}{2}$ is the identity matrix $I_{2}$. Thus the
monodromy representation of (\ref{89-1}) is reduced to $\rho:\pi
_{1}(  E_{\tau}\backslash \{  \pm p\}  ,q_{0})
\rightarrow SL(2,\mathbb{C})$. Let $\gamma_{\pm}\in \pi_{1}(  E_{\tau
}\backslash(\{ \pm p\} \cup E_{\tau}[2]),q_{0})  $ be a simple
loop encircling $\pm p$ counterclockwise respectively, and $\ell_{j}\in \pi
_{1}(  E_{\tau}\backslash(\{ \pm p\} \cup E_{\tau}[2]),q_{0})$, $j=1,2$, be two fundamental cycles of $E_{\tau}$ connecting
$q_{0}$ with $q_{0}+\omega_{j}$ such that $\ell_{j}$ does not intersect with
$L+\Lambda_{\tau}$ (here $L$ is the straight segment connecting $\pm p$) and
satisfies%
\[
\gamma_{+}\gamma_{-}=\ell_{1}\ell_{2}\ell_{1}^{-1}\ell_{2}^{-1}\text{ in }%
\pi_{1}(  E_{\tau}\backslash  \{  \pm p \}
,q_{0})  .
\]
Since the local exponents of (\ref{89-1}) at $\pm p$ are $
\{-\frac{1}{2}, \frac{3}{2}\}$ and $\pm p\not \in E_{\tau}[2]$ are apparent singularities, we
always have
$\rho(\gamma_{\pm})=-I_{2}$.
Denote by $N_{j}=\rho(\ell_j)$ the monodromy matrix along the loop $\ell_{j}$ of
(\ref{89-1}) with respect to any linearly independent solutions. Then $N_1N_2=N_2N_1$ and the monodromy group of (\ref{89-1}) is
generated by $\{-I_{2},N_{1},N_{2}\}$, i.e. is always \emph{abelian and so reducible}.
It is known (cf. \cite{CKL1}) that expect finitely many $A$'s for given $(\tau, p)$,
$N_1$ and $N_2$ can be diagonalized simultaneously, and more precisely, there exists
$(r,s)\in\mathbb{C}^2\setminus\frac12\mathbb{Z}^2$ such that
\[N_{1}=%
\begin{pmatrix}
e^{-2\pi is} & 0\\
0 & e^{2\pi is}%
\end{pmatrix},\quad N_{2}=%
\begin{pmatrix}
e^{2\pi ir} & 0\\
0 & e^{-2\pi ir}%
\end{pmatrix}.\]

Let
$U$ be an open subset of $\mathbb{H}$ such that $p(\tau)\not \in E_{\tau}[2]$
for any $\tau \in U$. Then we proved in \cite{Chen-Kuo-Lin0} that
$p(\tau)$ \emph{is a solution of EPVI$_{\mathbf{n}}$ if and only
if there exist $A(\tau)$ (and the corresponding $B(\tau)$ via (\ref{101}))
 such that GLE$(\mathbf{n},p(\tau),A(\tau),\tau)$ is monodromy preserving
as $\tau \in U$ deforms}. Furthermore, $(p(\tau),A(\tau))$ satisfies the following new Hamiltonian system%
\begin{equation}
\left \{
\begin{array}
[c]{l}%
\frac{dp(\tau)}{d\tau}=\frac{\partial \mathcal{H}}{\partial A}=\frac{-i}{4\pi
}(2A-\zeta(2p;\tau)+2p\eta_{1}(\tau))\\
\frac{dA(\tau)}{d\tau}=-\frac{\partial \mathcal{H}}{\partial p}=\frac{i}{4\pi
}\left(
\begin{array}
[c]{l}%
(2\wp(2p;\tau)+2\eta_{1}(\tau))A-\frac{3}{2}\wp^{\prime}(2p;\tau)\\
-\sum_{k=0}^{3}n_{k}(n_{k}+1)\wp^{\prime}(p+\frac{\omega_{k}}{2};\tau)
\end{array}
\right)
\end{array}
\right.  . \label{142-0}%
\end{equation}
where
$
\mathcal{H}  =\frac{-i}{4\pi}(B+2p\eta_{1}(\tau)A)$. In other words, this Hamiltonian system is equivalent to EPVI$_{\mathbf{n}}$.
We refer the reader to \cite{Chen-Kuo-Lin0}
for the more general statement and the proof.
Here we need to apply the following result \cite{Chen-Kuo-Lin}.

\begin{theorem}(\cite[Theorem 5.3]{Chen-Kuo-Lin}) \label{thm-II-8}
For $\mathbf{n}$, let $p^{\mathbf{n}}(\tau)$ be a
solution to EPVI$_{\mathbf{n}}$. Then the following hold:

\begin{itemize}
\item[(1)]  For any $\tau$ satisfying $p^{\mathbf{n}}
(\tau)\not \in E_{\tau}[2]$, the monodromy group of the
associated GLE$(\mathbf{n}$, $ p^{\mathbf{n}}(\tau), A^{\mathbf{n}}(\tau), \tau)$ is generated by
\begin{equation}
\rho(\gamma_{\pm})=-I_{2},\text{ }N_{1}=%
\begin{pmatrix}
e^{-2\pi is} & 0\\
0 & e^{2\pi is}%
\end{pmatrix}
\text{, }N_{2}=%
\begin{pmatrix}
e^{2\pi ir} & 0\\
0 & e^{-2\pi ir}%
\end{pmatrix}
 \label{II-101}%
\end{equation}
if and
only if $(r,s)  \in \mathbb{C}^{2}\backslash \frac
{1}{2}\mathbb{Z}^{2}$ and $p^{\mathbf{n}}(\tau)=p_{r,s}^{\mathbf{n}
}(\tau)$ in the sense of Remark \ref{identify}.

\item[(2)] $\wp(p^{\mathbf{n}}_{r_{1},s_{1}}(\tau);\tau)\equiv \wp
(p_{r_{2},s_{2}}^{\mathbf{n}}(\tau);\tau)\Longleftrightarrow(r_{1},s_{1})  \equiv \pm(r_{2},s_{2})  \operatorname{mod}$
$\mathbb{Z}^{2}$.
\end{itemize}
\end{theorem}

\subsection{Relation between PVI and $Z_{r,s}^{\mathbf{n}}(\tau)$}
In this subsection, we study the deep connection between PVI and the pre-modular form $Z_{r,s}^{\mathbf{n}}(\tau)$.
For $\mathbf{n}=(n_0,n_1,n_2,n_3)$, we define \begin{equation}\label{n-0-0}\mathbf{n}_0^{\pm}:=(n_0\pm 1,n_1,n_2,n_3).\end{equation}

\begin{lemma}\cite[Lemma 3.1]{Chen-Kuo-Lin0}\label{lem-3.1}
Fix any $\tau_{0}\in \mathbb{H}$ and $c_{n_0}^{2}\in\{\pm i\frac{2n_{0}+1}{2\pi}\}$. Then for any $h\in \mathbb{C}$, EPVI$_{\mathbf{n}}$ has a solution $p_{h}^{\mathbf{n}}(\tau)$ satisfying the following asymptotic
behavior%
\begin{align} \label{515-5}
p_{h}^{\mathbf{n}}(\tau)=&c_{n_0}(\tau-\tau_{0})^{\frac{1}{2}}%
(1+h(\tau-\tau_{0})\\
&+a(\tau-\tau_{0})^{2}+O((\tau-\tau_0)^3))\text{ as }\tau
\rightarrow \tau_{0},\nonumber
\end{align}
with
\[a=\frac{-h^2}{2}-\frac{(n_0+\frac12)^2g_2(\tau_0)}{240\pi^2}-\frac{1}{24\pi^2}\sum_{k=1}^3
(n_k+\tfrac12)^2\wp''(\tfrac{\omega_k}{2};\tau_0).\]
Moreover, these two $1$-parameter families of solutions give all
solutions $p^{\mathbf{n}}(\tau)$ of EPVI$_{\mathbf{n}}$ satisfying
$p^{\mathbf{n}}(\tau_{0})=0$.
\end{lemma}
Recall Remark \ref{identify} that we identify the solutions $p_{h}^{\mathbf{n}}(\tau)$ and $-p_{h}^{\mathbf{n}}(\tau)$, so (\ref{515-5}) gives
two $1$-parameter families (one family is given by $c_{n_0}^{2}= i\frac{2n_{0}+1}{2\pi}$ and the other by $c_{n_0}^{2}= -i\frac{2n_{0}+1}{2\pi}$) of solutions of EPVI$_{\mathbf{n}}$.

The first result of this subsection is
\begin{theorem} \label{thm-pole-z} Let $(r,s)\in\mathbb{C}^2\setminus\frac12\mathbb{Z}^2$ and $p_{r,s}^{\mathbf{n}
}(\tau)$ be a solution of EPVI$_{\mathbf{n}}$.
Suppose $p^{\mathbf{n}}_{r,s}(\tau_{0})=0$ for some $\tau_{0}\in \mathbb{H}$.
Then $p_{r,s}^{\mathbf{n}}(\tau)= p_{h}^{\mathbf{n}}(\tau)$ for
some $h\in \mathbb{C}$. Furthermore, by defining
\begin{equation}
B_{0}:=2\pi ic_{n_0}^{2}\left(  4\pi i h-\eta_{1}(\tau_{0})\right)
-\sum_{j=1}^{3}n_{j}(n_{j}+1)e_{j}(\tau_{0}), \label{kB}%
\end{equation}
the following hold.
\begin{itemize}
\item[(1)]If
$c_{n_0}^{2}=-i\frac{2n_{0}+1}{2\pi}$, then $Z_{r,s}^{\mathbf{n}_0^{+}}(\tau_0)=0$ and the monodromy of H$(\mathbf{n}_0^{+},B_{0}, \tau_0)$ is generated by
\begin{equation}
\rho(\ell_{1})=%
\begin{pmatrix}
e^{-2\pi is} & 0\\
0 & e^{2\pi is}%
\end{pmatrix}
,\text{ \  \  \ }\rho(\ell_{2})=%
\begin{pmatrix}
e^{2\pi ir} & 0\\
0 & e^{-2\pi ir}%
\end{pmatrix}.
\label{Mono0-1}%
\end{equation}
\item[(2)] If $c_{n_0}^{2}=i\frac{2n_{0}+1}{2\pi}$, then $Z_{r,s}^{\mathbf{n}_0^{-}}(\tau_0)=0$ and the monodromy of H$(\mathbf{n}_0^{-},B_{0}, \tau_0)$ is generated by (\ref{Mono0-1}).
\end{itemize}
\end{theorem}

\begin{proof}
The assertion $p_{r,s}^{\mathbf{n}}(\tau)= p_{h}^{\mathbf{n}}(\tau)$ for
some $h\in \mathbb{C}$ follows from Lemma \ref{lem-3.1}.
Furthermore, letting $B_0$ be given by \eqref{kB}, we proved in \cite[Theorem 3.1]{Chen-Kuo-Lin0} that
\begin{itemize}
\item[(1)] if $c_{n_0}%
^{2}=-i\frac{2n_{0}+1}{2\pi}$, then as $\tau\to \tau_0$, 
\begin{align*}
&\sum_{k=0}^{3}n_{k}(n_{k}+1)\wp(z+\tfrac{\omega_{k}}{2};\tau)+\tfrac{3}{4}
(\wp(z+p_{r,s}^{\mathbf{n}}(\tau);\tau)+\wp(z-p_{r,s}^{\mathbf{n}}(\tau);\tau))\\
&+A(\tau)(\zeta(z+p_{r,s}^{\mathbf{n}}(\tau);\tau)-\zeta(z-p_{r,s}^{\mathbf{n}}(\tau);\tau))+B(\tau)\\
&\to (n_0+1)(n_0+2)\wp(z;\tau_0)+\sum_{k=1}^{3}n_{k}(n_{k}+1)\wp(z+\tfrac{\omega_{k}}{2};\tau_0)+B_0,
\end{align*}
so the associated
GLE$(\mathbf{n}, p_{r,s}^{\mathbf{n}}(\tau), A(\tau), \tau)$
converges to H$(\mathbf{n}_0^{+},B_{0}, \tau_0)$;
\item[(2)] if $c_{n_0}%
^{2}=i\frac{2n_{0}+1}{2\pi}$, then as $\tau\to \tau_0$, 
\begin{align*}
&\sum_{k=0}^{3}n_{k}(n_{k}+1)\wp(z+\tfrac{\omega_{k}}{2};\tau)+\tfrac{3}{4}
(\wp(z+p_{r,s}^{\mathbf{n}}(\tau);\tau)+\wp(z-p_{r,s}^{\mathbf{n}}(\tau);\tau))\\
&+A(\tau)(\zeta(z+p_{r,s}^{\mathbf{n}}(\tau);\tau)-\zeta(z-p_{r,s}^{\mathbf{n}}(\tau);\tau))+B(\tau)\\
&\to (n_0-1)n_0\wp(z;\tau_0)+\sum_{k=1}^{3}n_{k}(n_{k}+1)\wp(z+\tfrac{\omega_{k}}{2};\tau_0)+B_0,
\end{align*}
so the associated
GLE$(\mathbf{n}, p_{r,s}^{\mathbf{n}}(\tau), A(\tau), \tau)$
converges to H$(\mathbf{n}_0^{-},B_{0}, \tau_0)$.
\end{itemize}
Recalling Theorem \ref{thm-II-8} that the monodromy of this
GLE$(\mathbf{n}, p_{r,s}^{\mathbf{n}}(\tau), A(\tau), \tau)$ is given by (\ref{II-101}), it follows from \cite[Theorem 6.2]{Chen-Kuo-Lin} that the monodromy of H$(\mathbf{n}_0^{+},B_{0}, \tau_0)$ for $c_{n_0}%
^{2}=-i\frac{2n_{0}+1}{2\pi}$ (resp. H$(\mathbf{n}_0^{-},B_{0}, \tau_0)$ for
$c_{n_0}^{2}=i\frac{2n_{0}+1}{2\pi}$) is given by (\ref{Mono0-1}). This, together with Theorem \ref{thm-premodular-1}, implies that $Z_{r,s}^{\mathbf{n}_0^{+}}(\tau_0)=0$ for $c_{n_0}
^{2}=-i\frac{2n_{0}+1}{2\pi}$ (resp. $Z_{r,s}^{\mathbf{n}_0^{-}}(\tau_0)=0$ for
$c_{n_0}^{2}=i\frac{2n_{0}+1}{2\pi}$).
\end{proof}

The converse statement of Theorem \ref{thm-pole-z} is also true, which  strongly suggests that $Z_{r,s}^{\mathbf{n}}(\tau)$ appears in the denominators of the expressions of both $\wp(p_{r,s}^{\mathbf{n}_0^+}(\tau);\tau)$ and $\wp(p_{r,s}^{\mathbf{n}_0^-}(\tau);\tau)$.

\begin{theorem} \label{thm-pole-z1} Let $(r,s)\in\mathbb{C}^2\setminus\frac12\mathbb{Z}^2$ such that $Z_{r,s}^{\mathbf{n}}(\tau_0)=0$ for some $\tau_{0}\in \mathbb{H}$, namely there is a unique $B_0\in\mathbb{C}$ such that the monodromy of H$(\mathbf{n},B_{0}, \tau_0)$ is generated by (\ref{Mono0-1}).
Then
\begin{itemize}
\item[(1)] $p_{r,s}^{\mathbf{n}_0^+}(\tau_0)=0$ and $p_{r,s}^{\mathbf{n}_0^+}(\tau)=p_{h}^{\mathbf{n}_0^+}(\tau)$ with $c_{n_0+1}^2=i\frac{2n_{0}+3}{2\pi}$ and $h$ satisfying
    \begin{equation}
B_{0}=2\pi ic_{{n}_0+1}^{2}\left(  4\pi i h-\eta_{1}(\tau_{0})\right)
-\sum_{j=1}^{3}n_{j}(n_{j}+1)e_{j}(\tau_{0}). \label{B1}%
\end{equation}

\item[(2)] $p_{r,s}^{\mathbf{n}_0^-}(\tau_0)=0$ and $p_{r,s}^{\mathbf{n}_0^-}(\tau)=p_{h}^{\mathbf{n}_0^-}(\tau)$ with $c_{n_0-1}^2=-i\frac{2n_{0}-1}{2\pi}$ and $h$ satisfying
    \begin{equation}
B_{0}=2\pi ic_{{n}_0-1}^{2}\left(  4\pi i h-\eta_{1}(\tau_{0})\right)
-\sum_{j=1}^{3}n_{j}(n_{j}+1)e_{j}(\tau_{0}). \label{B2}%
\end{equation}
\end{itemize}
\end{theorem}

\begin{proof} Note that the existence of $B_0$ follows from Theorem \ref{thm-premodular-1} and the uniqueness of $B_0$ follows from (\ref{fc-glo-uni}).

By defining $h$ in terms of $B_0$ via (\ref{B1}), Lemma \ref{lem-3.1} implies the existence of a solution $p_{h}^{\mathbf{n}_0^+}(\tau)$ of EPVI$_{\mathbf{n}_0^+}$ such that
\[p_{h}^{\mathbf{n}_0^+}(\tau)=c_{n_0+1}(\tau-\tau_{0})^{\frac{1}{2}}%
(1+h(\tau-\tau_{0})+O(\tau-\tau_{0})^{2})\text{ as }\tau
\rightarrow \tau_{0},\]
with $c_{n_0+1}^2=i\frac{2n_{0}+3}{2\pi}$. Then as mentioned in the proof of Theorem \ref{thm-pole-z}, the associated
GLE$(\mathbf{n}, p_{h}^{\mathbf{n}_{0}^{+}}(\tau), A(\tau), \tau)$
converges to H$(\mathbf{n},B_{0}, \tau_0)$. Since the monodromy of H$(\mathbf{n},B_{0}, \tau_0)$ is given by (\ref{Mono0-1}), \cite[Theorem 6.2]{Chen-Kuo-Lin} says that the monodromy of GLE$(\mathbf{n}, p_{h}^{\mathbf{n}_{0}^{+}}(\tau), A(\tau), \tau)$ is given by (\ref{II-101}). From here and Theorem \ref{thm-II-8}, we conclude that $p_{h}^{\mathbf{n}_0^+}(\tau)=p_{r,s}^{\mathbf{n}_0^+}(\tau)$. This proves the assertion (1), and the assertion (2) can be proved similarly.
\end{proof}

Given $\mathbf{n}=(n_0,n_1,n_2,n_3)$, we define
\begin{equation}
\mathbf{n}_k:=(n_{k,0},n_{k,1},n_{k,2},n_{k,3})=\begin{cases} (n_1,n_0,n_3,n_2)\quad\text{if }\;\, k=1,\\
(n_2,n_3,n_0,n_1)\quad\text{if }\;\, k=2,\\
(n_3,n_2,n_1,n_0)\quad\text{if }\;\, k=3.\end{cases}
\end{equation}
Then it follows from (\ref{124}) that
\begin{equation}\text{\it $p^{\mathbf{n}}(\tau)$ solves EPVI$_{\mathbf{n}}$ if and only if  $p^{\mathbf{n}}(\tau)-\tfrac{\omega_k}{2}$ solves EPVI$_{\mathbf{n}_k}$}.\end{equation}
Given $(r,s)$ we define
\begin{equation}\label{fc90}
(r_k,s_k):=\begin{cases} (r+\frac12,s)\quad\text{if }\;\, k=1,\\
(r,s+\frac12)\quad\text{if }\;\, k=2,\\
(r+\frac12,s+\frac12)\quad\text{if }\;\, k=3.\end{cases}
\end{equation}
This following result will play a crucial role in our proof of Theorem \ref{thm-rsbtau}.
\begin{lemma} \label{lem-3.4}
$p_{r_k,s_k}^{\mathbf{n}}(\tau)-\tfrac{\omega_k}{2}=p_{r,s}^{\mathbf{n}_k}(\tau)$ in the sense of Remark \ref{identify}.
\end{lemma}

\begin{proof}
Denote $p^{\mathbf{n}_k}(\tau):=p_{r_k,s_k}^{\mathbf{n}}(\tau)-\tfrac{\omega_k}{2}$, which is a solution of EPVI$_{\mathbf{n}_k}$. By Theorem \ref{thm-II-8}, to prove $p^{\mathbf{n}_k}=p^{\mathbf{n}_k}_{r,s}$ is equivalent to prove that the monodromy of the
associated GLE$(\mathbf{n}_k$, $ p^{\mathbf{n}_k}(\tau), A^{\mathbf{n}_k}(\tau), \tau)$ is given by (\ref{II-101}).

Theorem \ref{thm-II-8} says that the monodromy group of the
associated GLE$(\mathbf{n}$, $ p_{r_k,s_k}^{\mathbf{n}}(\tau), A^{\mathbf{n}}(\tau), \tau)$ is generated by
\[\rho(\gamma_{\pm})=-I_{2},\quad N_{1}=%
\begin{pmatrix}
e^{-2\pi is_k} & 0\\
0 & e^{2\pi is_k}%
\end{pmatrix},\quad N_{2}=%
\begin{pmatrix}
e^{2\pi ir_k} & 0\\
0 & e^{-2\pi ir_k}%
\end{pmatrix}.\]
More precisely, since the local exponents of GLE$(\mathbf{n}$, $ p_{r_k,s_k}^{\mathbf{n}}(\tau), A^{\mathbf{n}}(\tau), \tau)$ at $\pm p_{r_k,s_k}^{\mathbf{n}}(\tau)$ are $
\{-\frac{1}{2}, \frac{3}{2}\}$ and $\pm p_{r_k,s_k}^{\mathbf{n}}(\tau)\not \in E_{\tau}[2]$ are apparent singularities, it was proved in \cite{Chen-Kuo-Lin} that GLE$(\mathbf{n}$, $ p_{r_k,s_k}^{\mathbf{n}}(\tau), A^{\mathbf{n}}(\tau), \tau)$ has a linearly independent solutions of the form
\[y_1(z)=\Phi_{p_{r_k,s_k}^{\mathbf{n}}}(z)\hat{y}(z),\quad y_2(z)=\Phi_{p_{r_k,s_k}^{\mathbf{n}}}(z)\hat{y}(-z),\]
where
\[\Phi_{p}(z):=\frac{\sigma(z)}{\sqrt{\sigma(z-p)
\sigma(z+p)}},\]
and $\hat{y}(z)$ is meromorphic in $\mathbb{C}$ and satisfies the transformation law
\begin{equation}\label{eq-hat}\hat{y}(z+1)=e^{-2\pi i s_k}\hat{y}(z),\quad \hat{y}(z+\tau)=e^{2\pi i r_k}\hat{y}(z),\end{equation}
namely $\hat{y}(z)$ is elliptic of the second kind. Note that $\Phi_{p}(z)^2$ is an even elliptic function.

Now since $p^{\mathbf{n}_k}(\tau)=p_{r_k,s_k}^{\mathbf{n}}(\tau)-\frac{\omega_k}{2}$ and (recall $\eta_3=\eta_1+\eta_2$ and $\tau\eta_1-\eta_2=2\pi i$)
\[\frac{d}{d\tau}\frac{\omega_k}{2}=\frac{-i}{4\pi}(w_k\eta_1-\eta_k),\] it follows from the first equation of (\ref{142-0}) that the corresponding $A^{\mathbf{n}_k}(\tau)=A^{\mathbf{n}}(\tau)$. Then it is easy to see that
\[\tilde{y}_1(z):=y_1(z-\tfrac{\omega_k}{2})
=\Phi_{p_{r_k,s_k}^{\mathbf{n}}}(z-\tfrac{\omega_k}{2})\hat{y}(z-\tfrac{\omega_k}{2}),\]
\[\tilde{y}_2(z):=y_2(z-\tfrac{\omega_k}{2})
=\Phi_{p_{r_k,s_k}^{\mathbf{n}}}(z-\tfrac{\omega_k}{2})\hat{y}(-z+\tfrac{\omega_k}{2}),\]
are linearly independent solutions of GLE$(\mathbf{n}_k, p^{\mathbf{n}_k}(\tau), A^{\mathbf{n}_k}(\tau), \tau)$. It suffices to prove that the monodromy matrix of GLE$(\mathbf{n}_k$, $ p^{\mathbf{n}_k}(\tau), A^{\mathbf{n}_k}(\tau), \tau)$ with respect to $(\tilde{y}_1,\tilde{y}_2)$ are given by
(\ref{II-101}).

By the transformation law \eqref{tran-law} we have
\[\sigma(z-p^{\mathbf{n}_k}-{\omega_k})=-e^{-\eta_k(z-p^{\mathbf{n}_k}-\frac{\omega_k}{2})}\sigma(z-p^{\mathbf{n}_k}),\]
so
\begin{align*}
\Phi_{p_{r_k,s_k}^{\mathbf{n}}}(z-\tfrac{\omega_k}{2})
&=\frac{\sigma(z-\tfrac{\omega_k}{2})}
{\sqrt{\sigma(z-p^{\mathbf{n}_k}-{\omega_k})
\sigma(z+p^{\mathbf{n}_k})}}\\
&=\Phi_{p^{\mathbf{n}_k}}(z)\xi_k(z),
\end{align*}
where
\[\xi_k(z):=\epsilon_ke^{\frac{\eta_k}{2}(z-p^{\mathbf{n}_k}-\frac{\omega_k}{2})}
\frac{\sigma(z-\frac{\omega_k}{2})}{\sigma(z)},\quad \epsilon_k\in\{\pm i\}.\]
Again by the transformation law \eqref{tran-law}, a direct computation gives
\[\xi_k(z+1)=e^{\frac{\eta_k-\omega_k\eta_1}{2}}\xi_k(z)=
\begin{cases}
\xi_k(z)\quad\text{if }k=1,\\
-\xi_k(z)\quad \text{if }k=2,3,
\end{cases}\]
\[\xi_k(z+\tau)=e^{\frac{\tau\eta_k-\omega_k\eta_2}{2}}\xi_k(z)=
\begin{cases}
\xi_k(z)\quad\text{if }k=2,\\
-\xi_k(z)\quad \text{if }k=1,3.
\end{cases}\]
Now for GLE$(\mathbf{n}_k$, $ p^{\mathbf{n}_k}(\tau), A^{\mathbf{n}_k}(\tau), \tau)$, recalling that $\ell_{j}, j=1,2$, are two fundamental cycles of $E_{\tau}$ connecting
$q_{0}$ with $q_{0}+\omega_{j}$ such that $\ell_{j}$ does not intersect with
$L+\Lambda_{\tau}$ (here $L$ is the straight segment connecting $\pm p^{\mathbf{n}_k}$), it was proved in \cite[Lemma 2.2]{CKL1} that
\[\ell_j^*\Phi_{p^{\mathbf{n}_k}}(z)=\Phi_{p^{\mathbf{n}_k}}(z),\quad j=1,2,\]
where $\ell_j^*y(z)$ denotes the analytic continuation of $y(z)$ along $\ell_j$. Together these with $\tilde{y}_1(z)=\Phi_{p^{\mathbf{n}_k}}(z)\xi_k(z)\hat{y}(z-\tfrac{\omega_k}{2})$, (\ref{eq-hat}) and (\ref{fc90}), we finally obtain
\[\ell_1^*\tilde{y}_1(z)=\Phi_{p^{\mathbf{n}_k}}(z)\xi_k(z+1)\hat{y}(z-\tfrac{\omega_k}{2}+1)=e^{-2\pi i s}\tilde{y}_1(z),\]
\[ \ell_2^*\tilde{y}_1(z)=\Phi_{p^{\mathbf{n}_k}}(z)\xi_k(z+\tau)\hat{y}(z-\tfrac{\omega_k}{2}+\tau)=e^{2\pi i r}\tilde{y}_1(z).\]
Similarly, we obtain from $\tilde{y}_2(z)=\Phi_{p^{\mathbf{n}_k}}(z)\xi_k(z)\hat{y}(-z+\tfrac{\omega_k}{2})$ that
\[\ell_1^*\tilde{y}_2(z)=e^{2\pi i s}\tilde{y}_2(z),\quad \ell_2^*\tilde{y}_2(z)=e^{-2\pi i r}\tilde{y}_2(z).\]
In conclusion, the monodromy matrix of GLE$(\mathbf{n}_k$, $ p^{\mathbf{n}_k}(\tau), A^{\mathbf{n}_k}(\tau), \tau)$ with respect to $(\tilde{y}_1,\tilde{y}_2)$ are given by
(\ref{II-101}), and so Theorem \ref{thm-II-8} implies $p^{\mathbf{n}_k}(\tau)=p^{\mathbf{n}_k}_{r,s}(\tau)$ in the sense of Remark \ref{identify}. This completes the proof.
\end{proof}

\section{Simple zero property of $Z_{r,s}^{(\mathbf{n})}(\tau)$}
\label{sec-simplezero}

This section is denoted to the proof of Theorem \ref{thm-simplezero}.

\begin{proof}[Proof of Theroem \ref{thm-simplezero}] Fix $(r_0,s_0)\in\mathbb{C}^2\setminus\frac{1}{2}\mathbb{Z}^2$. Assume by contradiction that
$Z_{r_0,s_0}^{\mathbf{n}}(\cdot)$ has a zero $\tau_0$ of order $k\geq 2$. Then there is a small open neighborhood $V$ of $\tau_0$ such that $Z_{r_0,s_0}^{\mathbf{n}}(\tau)$ has no zeros in $\overline{V}\setminus\{\tau_0\}$.

We divide the proof into several steps.

{\bf Step 1.} We show that there is a small open neighborhood $U\subset\mathbb{C}^2\setminus\frac{1}{2}\mathbb{Z}^2$ of $(r_0, s_0)$ such that for any $(r,s)\in U$, $Z_{r,s}^{\mathbf{n}}(\cdot)$ has a zero $\tau(r,s)$ of order $k$ satisfying $\tau(r,s)\to \tau_0$ as $(r,s)\to (r_0, s_0)$ and $Z_{r,s}^{\mathbf{n}}(\cdot)$ has no zeros in $V\setminus\{\tau(r,s)\}$.

Since $Z_{r,s}^{\mathbf{n}}(\tau)$ is meromorphic in $\tau$, it follows from Rouch\'{e}'s theorem that there is a small open neighborhood $U\subset\mathbb{C}^2\setminus\frac{1}{2}\mathbb{Z}^2$ of $(r_0, s_0)$ such that for any $(r,s)\in U$, $Z_{r,s}^{\mathbf{n}}(\cdot)$ has exactly $k$ zeros
\[\tau_1(r,s),\cdots,\tau_k(r,s) \quad\text{\it up to multiplicity in $V$}\]
and $\tau_j(r,s)\to \tau_0$ for all $1\leq j\leq k$ as $(r,s)\to (r_0,s_0)$. On the other hand, we define
\[F_{r,s}(\tau):=\frac{1}{\wp(p_{r,s}^{\mathbf{n}_0^+}(\tau);\tau)},\]
which is meromorphic in $\tau$. Then by Theorem \ref{thm-pole-z1}-(1) and Lemma \ref{lem-3.1}, we see that $\tau_0$ is a {\it simple zero} of $F_{r_0,s_0}(\tau)$ and $ F_{r_0,s_0}(\tau)$ has no other zeros in $\overline{V}$. Again by Rouch\'{e}'s theorem, the zero number of $F_{r,s}(\tau)$ is also $1$ in $V$ for any $(r,s)\in U$ (by taking $U$ smaller if necessary). This, together with the fact that $\tau_{j}(r,s)\in V$ is a {\it simple zero} of of $F_{r,s}(\tau)$ for each $1\leq j\leq k$, implies
\[\tau_1(r,s)=\cdots=\tau_k(r,s),\]
namely $\tau(r,s):=\tau_1(r,s)$ is a zero of $Z_{r,s}^{\mathbf{n}}(\tau)$ of order $k$ for any $(r,s)\in U$.

{\bf Step 2. } We prove that $U\ni (r,s)\to \tau(r,s)$ is holomorphic. Consequently,
\begin{equation}\label{fc-fc1}
(\tfrac{\partial}{\partial r}Z_{r,s}^{\mathbf{n}})(\tau(r,s))=(\tfrac{\partial}{\partial s}Z_{r,s}^{\mathbf{n}})(\tau(r,s))=0, \quad\forall (r,s)\in U.
\end{equation}
Indeed, we define
\[G_{r,s}(\tau):=\frac{\partial^{k-1}}{\partial \tau^{k-1}}Z_{r,s}^{\mathbf{n}}(\tau),\]
which is meromorphic in $\tau$. By Step 1 we know that $\tau(r,s)$ is a simple zero of $G_{r,s}(\tau)$ for any $(r,s)\in U$, so the implicit function theorem yields that $U\ni (r,s)\to \tau(r,s)$ is holomorphic. Now since
$Z_{r,s}^{\mathbf{n}}(\tau(r,s))\equiv 0$, we obtain
\[(\tfrac{\partial}{\partial r}Z_{r,s}^{\mathbf{n}})(\tau(r,s))+(\tfrac{\partial}{\partial \tau}Z_{r,s}^{\mathbf{n}})(\tau(r,s))\tau_{r}(r,s)=0,\]
\[(\tfrac{\partial}{\partial s}Z_{r,s}^{\mathbf{n}})(\tau(r,s))+(\tfrac{\partial}{\partial \tau}Z_{r,s}^{\mathbf{n}})(\tau(r,s))\tau_{s}(r,s)=0,\]
where $\tau_r(r,s):=\frac{\partial{\tau(r,s)}}{\partial r}$ and $\tau_s(r,s):=\frac{\partial{\tau(r,s)}}{\partial s}$. Since Step 1 and $k\geq 2$ imply $(\tfrac{\partial}{\partial \tau}Z_{r,s}^{\mathbf{n}})(\tau(r,s))=0$, we obtain (\ref{fc-fc1}).

{\bf Step 3.} Recall Theorem \ref{thm-5A}-(3) and (\ref{kk-ll}) that
\begin{align*}Z_{r,s}^{\mathbf{n}}(\tau)=&W_{\mathbf{n}}(Z_{r,s}(\tau);r+s\tau,\tau)\\
\in &\mathbb{Q}[e_1(\tau),e_2(\tau
),e_3(\tau),\wp(r+s\tau;\tau),\wp^{\prime}(r+s\tau;\tau)][Z_{r,s}(\tau)].\end{align*}
Since \eqref{z-rs} implies 
\[\frac{\partial Z_{r,s}(\tau)}{\partial s}=-(\tau\wp+\eta_2)=-\tau(\wp+\eta_1)+2\pi i
=\tau \frac{\partial Z_{r,s}(\tau)}{\partial r}+2\pi i,\]
we easily obtain
\begin{equation}\label{eq-ex-taurs}\frac{\partial Z_{r,s}^{\mathbf{n}}(\tau)}{\partial s}=\tau \frac{\partial Z_{r,s}^{\mathbf{n}}(\tau)}{\partial r}+2\pi i \frac{\partial{W_{\mathbf{n}}}}{\partial X}(Z_{r,s}(\tau);r+s\tau,\tau).\end{equation}
From here and (\ref{fc-fc1}), we obtain
\begin{equation}
\label{fc-fc2} \frac{\partial{W_{\mathbf{n}}}}{\partial X}(Z_{r,s}(\tau(r,s));r+s\tau(r,s),\tau(r,s))\equiv 0, \quad\forall (r,s)\in U.
\end{equation}

{\bf Step 4.} Since $(r_0,s_0)\in\mathbb{C}^2\setminus\frac{1}{2}\mathbb{Z}^2$ and
$Z_{r_0,s_0}^{\mathbf{n}}(\tau_0)=0$, as mentioned in Theorem \ref{thm-pole-z1}, there is a unique $B_0\in\mathbb{C}$ such that the monodromy data of H$(\mathbf{n},B_{0}, \tau_0)$ is given by this $(r_0,s_0)$ and in particular $Q_{\mathbf{n}}(B_0;\tau_0)\neq 0$. Then there is $\varepsilon>0$ such that for any
$|B-B_0|<\varepsilon$, we have $Q_{\mathbf{n}}(B;\tau_0)\neq 0$, i.e. the monodromy of H$(\mathbf{n},B, \tau_0)$ is given by (\ref{Mono-1}) for some $(r,s)=(r(B),s(B))\notin\frac12\mathbb Z^2$ such that $(r(B),s(B))\to (r_0,s_0)$ as $B\to B_0$. Then Theorem \ref{thm-premodular} says
\begin{equation}\label{gvs}Z_{r(B),s(B)}^{\mathbf{n}}(\tau_0)=0.\end{equation}By taking $\varepsilon$ smaller we may assume $(r(B),s(B))\in U$. Then \eqref{gvs} and Step 1 together imply
\begin{equation}\label{fc-fc4}\tau(r(B),s(B))\equiv \tau_0,\quad \forall |B-B_0|<\varepsilon.\end{equation}

On the other hand, recalling the addition map (\ref{fc-fc3}):
\[\sigma
_{\mathbf{n}}:\overline{Y_{\mathbf{n}}(  \tau_0 )}\rightarrow E_{\tau_0},\]
the branch loci of which is a discrete set. So we can take $B$ satisfying $|B-B_0|<\varepsilon$ such that
\[\sigma:=r(B)+s(B)\tau_0 \notin E_{\tau_0}[2]\;\;\text{\it is outside the branch loci of }\sigma
_{\mathbf{n}}.\]
Then Theorem \ref{thm-5A} says that $W_{\mathbf{n}}(\cdot; \sigma,\tau_0)$ has $\frac12\sum_{k}n_k(n_k+1)$ distinct roots and so $Z_{r(B), s(B)}(\tau_0)$ is a {\it simple zero} of $W_{\mathbf{n}}(\cdot; \sigma,\tau_0)$, i.e.
\[\frac{\partial W_{\mathbf{n}}}{\partial X}(Z_{r(B), s(B)}(\tau_0); \sigma,\tau_0)\neq 0.\]
However, (\ref{fc-fc2}) and (\ref{fc-fc4}) imply
\[\frac{\partial W_{\mathbf{n}}}{\partial X}(Z_{r(B), s(B)}(\tau_0); \sigma,\tau_0)= 0,\]
clearly a contradiction.

Therefore, any zero $\tau_0$ of $Z_{r_0,s_0}^{\mathbf{n}}(\cdot)$ must be simple. This completes the proof.
\end{proof}

\section{Proof of Theorem \ref{thm-rsbtau}: the special cases}

In this section, we prove Theorem \ref{thm-rsbtau} for the first two Lam\'{e} case $(n,0,0,0)$ with $n=1,2$. We recall the following formulas (see e.g. \cite{YB}):
\begin{align*}
\frac{\partial}{\partial \tau}\zeta(  z;\tau)  =\frac{i}{4\pi
}\Big[&\wp^{\prime}(z;\tau)  +2(  \zeta(  z;\tau)
-z\eta_{1}(\tau))  \wp(z;\tau)\\
&+2\eta_{1}(\tau)\zeta (z;\tau)  -\frac{1}{6}zg_{2}(\tau)\Big],
\end{align*}
\begin{align}\label{eq-deri-wp}
\frac{\partial}{\partial \tau}\wp(z;\tau)  =\frac{-i}{4\pi}\Big[
&2(  \zeta(  z;\tau)  -z\eta_{1}(  \tau))
\wp^{\prime}(  z;\tau) \\
&+4(  \wp(  z;\tau)  -\eta_{1}(\tau))  \wp(
z;\tau)  -\frac{2}{3}g_{2}(\tau)
\Big],\nonumber
\end{align}
\begin{align*}
\frac{\partial}{\partial \tau}\wp^{\prime}(z;\tau)  =\frac
{-i}{4\pi}\Big[
&6(  \wp(  z;\tau)  -\eta_{1}(\tau))  \wp^{\prime}(
z;\tau) \\
&+(  \zeta(z;\tau)  -z\eta_{1}(\tau))
(  12\wp^{2}(  z;\tau)  -g_{2}(\tau))
\Big]  ,
\end{align*}
\[
\frac{d}{d\tau}\eta_{1}(  \tau)  =\frac{i}{24\pi}\left[  12\eta
_{1}(\tau)^{2}-g_{2}(\tau)  \right],
\]
\[\wp''(z;\tau)=\frac12[12\wp(z;\tau)^2-g_2(\tau)].\]
By applying these formulas and
\begin{equation}\label{fc-20}
Z_{r,s}(\tau)=\zeta(r+s\tau;\tau)-(r+s\tau)\eta_1(\tau)+2\pi i s,\end{equation}
a direct computation leads to 
\begin{align}\label{fc-21}\frac{\partial Z_{r,s}(\tau)}{\partial\tau}
&=-s\wp+\left(\tfrac{\partial}{\partial \tau}\zeta(z;\tau)\right)\Big|_{z=r+s\tau}-s\eta_1-\frac{i(r+s\tau)}{24\pi}[  12\eta
_{1}^{2}-g_{2}]\\
&=\frac{i}{4\pi}\wp'+\frac{i}{2\pi}(\wp+\eta_1)Z,\nonumber\end{align}
\begin{align}\label{fc-22}
\frac{\partial\wp(r+s\tau;\tau)}{\partial\tau}&=s\wp'+\left(\tfrac{\partial}{\partial \tau}\wp(z;\tau)\right)\Big|_{z=r+s\tau}\\
&=\frac{-i}{2\pi}[Z\wp'+2\wp^2-2\wp\eta_1
-\tfrac{1}{3}g_2],\nonumber
\end{align}
\begin{align}\label{fc-22-1}
\frac{\partial\wp'(r+s\tau;\tau)}{\partial\tau}&=s\wp''+\left(\tfrac{\partial}{\partial \tau}\wp'(z;\tau)\right)\Big|_{z=r+s\tau}\\
&=\frac{-i}{4\pi}[Z(12\wp^2-g_2)+6(\wp-\eta_1)\wp'],\nonumber
\end{align}
where as before, we write
\[Z_{r,s}(\tau)=Z,\quad\wp(r+s\tau;\tau)=\wp,\quad\wp'(r+s\tau;\tau)=\wp'\]
freely for convenience when there is no confusion arising.

\subsection{The case $n=1$} For this simplest case, $Z_{r,s}^{(1,0,0,0)}(\tau)=Z_{r,s}(\tau)=Z$.

By Theorem \ref{thm-premodular-1}, the monodromy data of
\begin{equation}
\label{eq-lame1}y''=[2\wp(z;\tau)+B]y(z)
\end{equation}
is $(r,s)\notin \frac12\mathbb{Z}^2$ if and only if
\[Z_{r,s}(\tau)=0,\quad B=\wp(r+s\tau;\tau).\]
Since $\tau$ is a simple zero of $Z_{r,s}(\cdot)$, it follows from the implicit function theorem that $\tau=\tau(r,s)$ is holomorphic in $(r,s)\in U$, where $U$ is a small open subset in $\mathbb{C}^2\setminus\frac12\mathbb{Z}^2$.
By $Z_{r,s}(\tau(r,s))=0$ and (\ref{fc-20})-(\ref{fc-21}), we have
\[\tau_r:=\frac{\partial\tau(r,s)}{\partial r}=-\frac{\frac{\partial Z_{r,s}(\tau)}{\partial r}}{\frac{\partial Z_{r,s}(\tau)}{\partial \tau}}=\frac{\wp+\eta_1}{\frac{i}{4\pi}\wp'},\]
\[\tau_s:=\frac{\partial\tau(r,s)}{\partial s}=-\frac{\frac{\partial Z_{r,s}(\tau)}{\partial s}}{\frac{\partial Z_{r,s}(\tau)}{\partial \tau}}=\frac{\tau(\wp+\eta_1)-2\pi i}{\frac{i}{4\pi}\wp'}
=\tau \tau_r-\frac{8\pi^2}{\wp'}.\]
Similarly by $B=\wp(r+s\tau(r,s);\tau(r,s))$ we have
\[B_r:=\frac{\partial B}{\partial r}=\wp'+\frac{\partial \wp}{\partial \tau}\tau_r,\quad
B_s:=\frac{\partial B}{\partial s}=\tau \wp'+\frac{\partial \wp}{\partial \tau}\tau_s.\]
Therefore,
\[\tau_rB_s-\tau_sB_r=(\tau \tau_r-\tau_s)\wp'=8\pi^2,\]
i.e. $d\tau\wedge dB=8\pi^2dr\wedge ds$. This proves Theorem \ref{thm-rsbtau} for the case $(1,0,0,0)$.

\subsection{The case $n=2$.}
For this case,
\[Z^{(2)}_{r,s}(\tau):=Z_{r,s}^{(2,0,0,0)}=Z^3-3\wp Z-\wp'.\]
For the corresponding Lam\'{e} equation
\begin{equation}
\label{eq-lame2}
y''(z)=[6\wp(z;\tau)+B]y(z),
\end{equation}
the well-known associated spectral polynomial $Q_2(B;\tau):=Q_{(2,0,0,0)}(B;\tau)$ is
given by
\begin{equation}\label{eq-spec-n=2}
Q_2(B;\tau)=(B^2-3g_2)(B^3-\tfrac{9}{4}g_2B+\tfrac{27}{4}g_3).
\end{equation}
Define $\pm C$ by $C^2=Q_2(B;\tau)$.
Then the monodromy data of (\ref{eq-lame2}) is $(r,s)\notin\frac12\mathbb{Z}^2$ if and only if
$Z_{r,s}^{(2)}(\tau)=0$
and the following formulas hold (see \cite[Theorem 5.3, Example 5.8]{LW2})
\begin{equation}\label{fc-26}
\wp=\wp(r+s\tau;\tau)=\frac{B^3+27g_3}{9(B^2-3g_2)},
\end{equation}
\begin{equation}\label{fc-27}
\wp'=\wp'(r+s\tau;\tau)=C\frac{2(B^3-9g_2B-54g_3)}{27(B^2-3g_2)^2},
\end{equation}
\begin{equation}\label{fc-28}
Z=Z_{r,s}(\tau)=C\frac{2}{3(B^2-3g_2)}.
\end{equation}
Clearly $(-r,-s)$ corresponds to $-C$ in these formulas. By (\ref{fc-26})-(\ref{fc-28}), a direct computation gives
\begin{equation}\label{fc-29}
B=3(Z^2-\wp).
\end{equation}
On the other hand, by (\ref{fc-21})-(\ref{fc-22-1}) and $\wp''=6\wp^2-\frac{g_2}{2}$, a direct computation gives
\begin{align}\label{fc-33}
\frac{\partial Z^{(2)}_{r,s}(\tau)}{\partial\tau}
&=3(Z^2-\wp)\frac{\partial Z_{r,s}(\tau)}{\partial\tau}-3Z\frac{\partial \wp(r+s\tau;\tau)}{\partial\tau}-\frac{\partial \wp'(r+s\tau;\tau)}{\partial\tau}\nonumber\\
&=
\frac{3i(\wp+\eta_1)}{2\pi}Z_{r,s}^{(2)}(\tau)+
\frac{3i}{4\pi}[3\wp'Z^2+(12\wp^2-g_2)Z+3\wp\wp'],
\end{align}
\begin{align}\label{fc-34}
\frac{\partial Z^{(2)}_{r,s}(\tau)}{\partial r}&=3(Z^2-\wp)\frac{\partial Z_{r,s}(\tau)}{\partial r}-3Z\frac{\partial \wp(r+s\tau;\tau)}{\partial r}-\frac{\partial \wp'(r+s\tau;\tau)}{\partial r}
\nonumber\\&=
-3(Z^2-\wp)(\wp+\eta_1)-(3Z\wp'+6\wp^2-\tfrac{g_2}{2}),\end{align}
\begin{align}\label{fc-35}
\frac{\partial Z^{(2)}_{r,s}(\tau)}{\partial s}&=3(Z^2-\wp)\frac{\partial Z_{r,s}(\tau)}{\partial s}-3Z\frac{\partial \wp(r+s\tau;\tau)}{\partial s}-\frac{\partial \wp'(r+s\tau;\tau)}{\partial s}
\nonumber\\&=\tau \frac{\partial Z^{(2)}_{r,s}(\tau)}{\partial r}+6\pi i (Z^2-\wp).
\end{align}

Now by $Z_{r,s}^{(2)}(\tau)=0$ and the implicit function theorem, $\tau=\tau(r,s)$ is holomorphic in $(r,s)\in U$, where $U$ is a small open subset in $\mathbb{C}^2\setminus\frac12\mathbb{Z}^2$.
Consequently, inserting (\ref{fc-26})-(\ref{fc-29}) into \eqref{fc-33}-\eqref{fc-35} leads to
\begin{align*}
\frac{\partial Z^{(2)}_{r,s}(\tau)}{\partial\tau}=
\frac{3i}{4\pi}[3\wp'Z^2+(12\wp^2-g_2)Z+3\wp\wp']=\frac{iC}{6\pi},
\end{align*}
\begin{align*}
\frac{\partial Z^{(2)}_{r,s}(\tau)}{\partial r}&=-3(Z^2-\wp)(\wp+\eta_1)-(3Z\wp'+6\wp^2-\tfrac{g_2}{2})\\
&=-\tfrac{1}{3}(B^2+3\eta_1B-\tfrac{3}{2}g_2),
\end{align*}
and so
\begin{align}\label{fc-38}
\tau_r=-\frac{\frac{\partial Z^{(2)}_{r,s}(\tau)}{\partial r}}{\frac{\partial Z^{(2)}_{r,s}(\tau)}{\partial \tau}}=-2\pi i \frac{B^2+3\eta_1B-\frac{3}{2}g_2}{C},
\end{align}
\begin{align}\label{fc-39}
\frac{\tau_s}{\tau_r}=\frac{\frac{\partial Z^{(2)}_{r,s}(\tau)}{\partial s}}{\frac{\partial Z^{(2)}_{r,s}(\tau)}{\partial r}}=\tau +\frac{6\pi i (Z^2-\wp)}{\frac{\partial Z^{(2)}_{r,s}(\tau)}{\partial r}}=\tau-\frac{6\pi i B}{B^2+3\eta_1B-\frac{3}{2}g_2}.
\end{align}
On the other hand, (\ref{fc-29}) gives
\begin{align}\label{fc-39--1}\frac{B_r}{3}&=\frac{\partial}{\partial r}\Big(Z_{r,s}(\tau(r,s))^2-\wp(r+s\tau(r,s);\tau(r,s))\Big)\\
&=-(2Z(\wp+\eta_1)+\wp')+(2Z\tfrac{\partial Z_{r,s}(\tau)}{\partial \tau}-\tfrac{\partial \wp(r+s\tau;\tau)}{\partial \tau})\tau_r.\nonumber\end{align}
Inserting (\ref{fc-21})-(\ref{fc-22}) and (\ref{fc-26})-(\ref{fc-28}) into this formula and by a direct computation, we obtain
\[2Z(\wp+\eta_1)+\wp'=\frac{2C}{9}\frac{B+6\eta_1}{B^2-3g_2},\]
\begin{align}\label{fc-39--2}&2Z\frac{\partial Z_{r,s}(\tau)}{\partial \tau}-\frac{\partial \wp(r+s\tau;\tau)}{\partial \tau}\\
=&\frac{i}{2\pi}Z\left[\wp'+2(\wp+\eta_1)Z\right]+\frac{i}{2\pi}[Z\wp'+2\wp^2-2\wp\eta_1
-\tfrac{1}{3}g_2]\nonumber\\
=&\frac{i}{9\pi}(B^2+3\eta_1B-\tfrac{3}{2}g_2),\nonumber\end{align}
and so
\begin{align}\label{fc-40}
\frac{B_r}{3}=-\frac{2C}{9}\frac{B+6\eta_1}{B^2-3g_2}
+\frac{i}{9\pi}(B^2+3\eta_1B-\tfrac{3}{2}g_2)\tau_r.
\end{align}
Similarly, by \eqref{fc-28} and \eqref{fc-39}-\eqref{fc-39--2}, we also have
\begin{align}\label{fc-41}\frac{B_s}{3}&=\frac{\partial}{\partial s}\Big(Z_{r,s}(\tau(r,s))^2-\wp(r+s\tau(r,s);\tau(r,s))\Big)\\&=4\pi iZ-\tau(2Z(\wp+\eta_1)+\wp')+\Big(2Z\frac{\partial Z_{r,s}(\tau)}{\partial \tau}-\frac{\partial \wp(r+s\tau;\tau)}{\partial \tau}\Big)\tau_s\nonumber\\
&=4\pi iZ+\tau \frac{B_r}{3}-\Big(2Z\frac{\partial Z_{r,s}(\tau)}{\partial \tau}-\frac{\partial \wp(r+s\tau;\tau)}{\partial \tau}\Big)\frac{6\pi i B}{B^2+3\eta_1B-\frac{3}{2}g_2}\tau_r\nonumber\\
&=\frac{8\pi i C}{3(B^2-3g_2)}+\tau\frac{B_r}{3}+\frac{2B}{3}\tau_r.\nonumber
\end{align}
Therefore, by (\ref{fc-38})-(\ref{fc-39}) and \eqref{fc-40}-\eqref{fc-41}, we have
{\allowdisplaybreaks
\begin{align*}
\tau_rB_s-\tau_sB_r&=\det\begin{pmatrix}\tau_r& \tau_r(\tau-\frac{6\pi i B}{B^2+3\eta_1B-\frac{3}{2}g_2})\\
B_r& \tau B_r+2B\tau_r+\frac{8\pi i C}{(B^2-3g_2)}
\end{pmatrix}\\
&=\tau_r\det\begin{pmatrix}1& -\frac{6\pi i B}{B^2+3\eta_1B-\frac{3}{2}g_2}\\
B_r& 2B\tau_r+\frac{8\pi i C}{(B^2-3g_2)}
\end{pmatrix}\\
&=\tau_r\Big(2B\tau_r+\frac{8\pi i C}{(B^2-3g_2)}+\frac{6\pi i B}{B^2+3\eta_1B-\frac{3}{2}g_2} B_r\Big)\\
&=\tau_r \Big(\frac{8\pi i C}{(B^2-3g_2)}-\frac{4\pi i C B(B+6\eta_1)}{(B^2-3g_2)(B^2+3\eta_1B-\frac{3}{2}g_2)}\Big)\\
&=\frac{4\pi i \tau_r C}{B^2+3\eta_1B-\frac{3}{2}g_2}=8\pi^2,
\end{align*}
}%
i.e. $d\tau\wedge dB=8\pi^2dr\wedge ds$. This proves Theorem \ref{thm-rsbtau} for the case $(2,0,0,0)$.

\subsection{A remark for the general case}

For general $\mathbf{n}$, it is impossible to prove Theorem \ref{thm-rsbtau} via direct computations. Instead, we will develop an induction approach to prove Theorem \ref{thm-rsbtau} in Section \ref{sec-induction}. Here we note that the following general result holds.

\begin{theorem}\label{thm-taurs} Fix $\mathbf{n}$. Then there exist rational functions
\[R_0(B;\tau), R_1(B;\tau)\in \mathbb{Q}[\eta_1,e_1,e_2,e_3,g_2,g_3](B)=\mathbb{Q}[\eta_1,e_1,e_2,e_3](B)\]
such that
\[\tau_r=\frac{\pi i}{C}R_0(B;\tau),\quad \frac{\tau_s}{\tau_r}=\tau+\pi i R_1(B;\tau).\]
In particular, for the Lam\'{e} case $\mathbf{n}=(n,0,0,0)$, there holds
\begin{equation}\label{eq-ex-llame}R_0(B;\tau), R_1(B;\tau)\in \mathbb{Q}[\eta_1,g_2,g_3](B).\end{equation}
\end{theorem}

\begin{proof}
By $e_k(\tau)=\wp(\frac{\omega_k}{2};\tau)$ and (\ref{eq-deri-wp}) we have
\begin{equation}\label{derivativek1}
e_{k}^{\prime}(\tau)=\frac{i}{\pi}\Big[  \frac{1}{6}g_{2}(\tau)+\eta
_{1}(\tau)e_{k}(\tau)-e_{k}(\tau)^{2}\Big].
\end{equation}
Recall Theorem \ref{thm-5A}-(3) and (\ref{kk-ll}) that
\begin{align*}Z_{r,s}^{\mathbf{n}}(\tau)=&W_{\mathbf{n}}(Z_{r,s}(\tau);r+s\tau,\tau)\\
\in &\mathbb{Q}[e_1(\tau),e_2(\tau
),e_3(\tau),\wp(r+s\tau;\tau),\wp^{\prime}(r+s\tau;\tau)][Z_{r,s}(\tau)].\end{align*}
Clearly it follows from
\[\frac{\partial \wp(r+s\tau;\tau)}{\partial r}=\wp', \;\frac{\partial \wp'(r+s\tau;\tau)}{\partial r}=\wp''=6\wp^2-g_2/2,\] \[\frac{\partial Z_{r,s}(\tau)}{\partial r}=-\wp-\eta_1,\]
that
\begin{align*}
\frac{\partial Z_{r,s}^{\mathbf{n}}(\tau)}{\partial r}\in \mathbb{Q}[e_1,e_2,e_3,\eta_1,\wp,\wp^{\prime}][Z_{r,s}(\tau)].
\end{align*}
Similarly we see from (\ref{fc-21})-(\ref{fc-22-1}) and (\ref{derivativek1}) that
\begin{align*}
\frac{\partial Z_{r,s}^{\mathbf{n}}(\tau)}{\partial \tau}\in \frac{i}{\pi}\times\mathbb{Q}[e_1,e_2,e_3,\eta_1,\wp,\wp^{\prime}][Z_{r,s}(\tau)],
\end{align*}
and so
\begin{equation}\label{eq-ex-taur}\tau_r=-\frac{\frac{\partial Z_{r,s}^{\mathbf{n}}(\tau)}{\partial r}}{\frac{\partial Z_{r,s}^{\mathbf{n}}(\tau)}{\partial \tau}}\in i\pi\times\mathbb{Q}[e_1,e_2,e_3,\eta_1,\wp,\wp^{\prime}](Z_{r,s}(\tau)).\end{equation}

On the other hand, it was proved by Takemura \cite[Theorem 2.3]{Takemura4} that there are
rational functions
\begin{equation}\label{eq-ex-llame1}\tilde{R}_1(B),\tilde{R}_2(B),\tilde{R}_3(B)\in \mathbb{Q}[e_1,e_2,e_3](B)\end{equation}
such that
\begin{equation}\label{eq-ex-zn}Z_{r,s}(\tau)=C \,\tilde{R}_1(B),\; \wp'(r+s\tau;\tau)=C\,\tilde{R}_3(B),\; \wp(r+s\tau;\tau)=\tilde{R}_2(B).\end{equation}
Inserting these into (\ref{eq-ex-taur}) leads to the existence of rational functions $R_0(B)=R_0(B;\tau)$, $\hat{R}_0(B)=\hat{R}_0(B;\tau)\in \mathbb{Q}[e_1,e_2,e_3,\eta_1](B)$ such that (note $C^2=Q_{\mathbf{n}}(B)$)
\[\frac{\tau_r}{\pi i}=\frac{R_0(B)}{C}+\hat{R}_0(B).\]
Since $(\tau, B)\to \pm (r,s)$ leads to $\tau(r,s)=\tau(-r,-s)$ and
\begin{equation}\label{ffrr}(r,s)\longleftrightarrow C,\quad (-r,-s)\longleftrightarrow -C \quad (\text{see} (\ref{eq-ex-zn})),\end{equation}
we have $\tau_r(r,s)=-\tau_r(-r,-s)$, i.e.
\[\frac{R_0(B)}{C}+\hat{R}_0(B)=-\Big(\frac{R_0(B)}{-C}+\hat{R}_0(B)\Big),\]
 and so $\hat{R}_0(B)\equiv 0$. This proves
\[\frac{\tau_r}{\pi i}=\frac{R_0(B)}{C}.\]

Similarly, we recall (\ref{eq-ex-taurs}) that
\[\frac{\tau_s}{\tau_r}=\frac{\frac{\partial Z_{r,s}^{\mathbf{n}}(\tau)}{\partial s}}{\frac{\partial Z_{r,s}^{\mathbf{n}}(\tau)}{\partial r}}=\tau +\pi i \frac{2\frac{\partial{W_{\mathbf{n}}}}{\partial X}(Z_{r,s}(\tau);r+s\tau,\tau)}{\frac{\partial Z_{r,s}^{\mathbf{n}}(\tau)}{\partial r}},\]
with
\begin{align*}
\frac{2\frac{\partial{W_{\mathbf{n}}}}{\partial X}(Z_{r,s}(\tau);r+s\tau,\tau)}{\frac{\partial Z_{r,s}^{\mathbf{n}}(\tau)}{\partial r}}\in \mathbb{Q}[e_1,e_2,e_3,\eta_1,\wp,\wp^{\prime}](Z_{r,s}(\tau)),
\end{align*}
so there are rational functions $R_1(B)=R_1(B;\tau)$, $\hat{R}_1(B)=\hat{R}_1(B;\tau)\in \mathbb{Q}$ $ [e_1,e_2,e_3,\eta_1](B)$ such that
\[\frac{2\frac{\partial{W_{\mathbf{n}}}}{\partial X}(Z_{r,s}(\tau);r+s\tau,\tau)}{\frac{\partial Z_{r,s}^{\mathbf{n}}(\tau)}{\partial r}}
=\frac{\hat{R}_1(B)}{C}+R_1(B).\]
By $\frac{\tau_s}{\tau_r}(-r,-s)=\frac{\tau_s}{\tau_r}(r,s)$ and \eqref{ffrr} we obtain $\hat{R}_1(B)\equiv 0$, so we have
\[\frac{\tau_s}{\tau_r}=\tau+\pi i R_1(B).\]

Finally, for the Lam\'{e} case $\mathbf{n}=(n,0,0,0)$, it was proved in \cite{LW2} that
\begin{align*}Z_{r,s}^{\mathbf{n}}(\tau)
\in \mathbb{Q}[g_2(\tau),g_3(\tau
),\wp(r+s\tau;\tau),\wp^{\prime}(r+s\tau;\tau)][Z_{r,s}(\tau)],\end{align*}
and (\ref{eq-ex-llame1}) can be improved as
\[\tilde{R}_1(B),\tilde{R}_2(B),\tilde{R}_3(B)\in \mathbb{Q}[g_2,g_3](B).\]
Therefore, the above argument actually implies (\ref{eq-ex-llame}).
This completes the proof.
\end{proof}

\section{Proof of Theorem \ref{thm-rsbtau}: An induction approach}

\label{sec-induction}

This section is devoted to the proof of Theorem \ref{thm-rsbtau} for general $\mathbf{n}$, which can not be proved by direct computations due to the lack of explicit expressions of $Z_{r,s}^{\mathbf{n}}(\tau)$, and new ideas are needed. We will introduce an induction approach to overcome this difficulty.

\subsection{The linearized equation of EPVI}

Fix any $\mathbf{n}$. Recall Lemma \ref{lem-2.3} that $\wp(p_{r,s}^{\mathbf{n}}(\tau);\tau)$ depends meromorphically on $(r,s)\in\mathbb{C}^2\setminus\frac12\mathbb{Z}^2$.
Thus
\begin{equation}
Y_{1;r,s}^{\mathbf{n}}(\tau):=\frac{\partial p_{r,s}^{\mathbf{n}}(\tau)}{\partial r},\quad
Y_{2;r,s}^{\mathbf{n}}(\tau):=\frac{\partial p_{r,s}^{\mathbf{n}}(\tau)}{\partial s}
\end{equation}
are well-defined and solve the linearized equation of EPVI$_{\mathbf{n}}$ as functions of $\tau$:
\begin{equation}
\frac{d^{2}}{d\tau^{2}}Y(\tau)=\bigg[\frac{-1}{8\pi^{2}}\sum_{k=0}^{3}(n_k+\tfrac12)^2
\wp''\left( p_{r,s}^{\mathbf{n}}(\tau)+\tfrac{\omega_{k}}{2};
\tau \right)\bigg] Y(\tau) . \label{linear-124}%
\end{equation}
Consequently, the Wronskian
\begin{equation}
W_{\mathbf{n}}(r,s):=\frac{dY_{1;r,s}^{\mathbf{n}}(\tau)}{d\tau}Y_{2;r,s}(\tau)
-\frac{dY_{2;r,s}^{\mathbf{n}}(\tau)}{d\tau}Y_{1;r,s}(\tau)
\end{equation}
is independent of $\tau$ and meromorphic in $(r,s)\in\mathbb{C}^2\setminus\frac12\mathbb{Z}^2$.

The following result is the key observation of proving Theorem \ref{thm-rsbtau}.

\begin{lemma}\label{lem-6.1} Fix $\mathbf{n}$, the following assertions are equivalent.
\begin{itemize}
\item[(1)] $W_{\mathbf{n}}(r,s)\equiv -1$.
\item[(2)]
For $\mathbf{n}_0^{-}=(n_0-1,n_1,n_2,n_3)$ with $n_0\geq 1$, the map $\varphi_{\mathbf{n}_{0}^{-1}}: (\tau, B)\mapsto (r,s)$ satisfies
\begin{equation}d\tau\wedge dB=8\pi^2 dr\wedge ds.\label{fc-43}\end{equation}
\item[(3)] For $\mathbf{n}_0^{+}=(n_0+1,n_1,n_2,n_3)$, the map $\varphi_{\mathbf{n}_0^{+}}: (\tau,B)\to (r,s)$ satisfies (\ref{fc-43}).
\end{itemize}
\end{lemma}

\begin{proof}
(1)$\Leftrightarrow$(2).
Fix any $(\tau_0,B_0)$ satisfying $Q_{\mathbf{n}_0^-}(B_0;\tau_0)\neq 0$ and let $(r_0,s_0)=\varphi_{\mathbf{n}_0^{-}}(\tau_0,B_0)\notin \frac12\mathbb{Z}^2$ be the monodromy data of H$(\mathbf{n}_0^{-},B_0,\tau_0)$, i.e. $Z_{r_0,s_0}^{\mathbf{n}_0^-}(\tau_0)=0$. Then there is a small open neighborhood $U\subset \mathbb{C}^2\setminus\frac12\mathbb{Z}^2$ of $(r_0,s_0)$ such that $Z_{r,s}^{\mathbf{n}_0^-}(\cdot)$ has a unique zero $\tau=\tau(r,s)$ in a small open neighborhood $V\subset \mathbb{H}$ of $\tau_0$ for any $(r,s)\in U$. Let $B=B(r,s)$ be the unique $B$ such that $(r,s)=\varphi_{\mathbf{n}_0^{-}}(\tau(r,s),B(r,s))$ is the monodromy data of H$(\mathbf{n}_0^-,B(r,s),\tau(r,s))$.
Then by Theorem \ref{thm-pole-z1}-(1) and Lemma \ref{lem-3.1}, we have
\begin{align}
p_{r,s}^{\mathbf{n}}(\tilde{\tau})=&c_{n_0}(\tilde{\tau}-\tau(r,s))^{\frac{1}{2}}\nonumber\\
&[1+h(\tilde{\tau}-\tau(r,s))+a(\tilde{\tau}-\tau(r,s))^{2}+O((\tilde{\tau}-\tau(r,s))^3)] \label{fc-44}
\end{align}
for $\tilde{\tau}$ near $\tau(r,s)$,
where
\begin{equation}\label{fc-46}
h=h(r,s)=\frac{B(r,s)+\vartheta_{\mathbf{n}}(\tau(r,s))}{-8\pi^2 c_{n_0}^2},
\end{equation}
\[
\vartheta_{\mathbf{n}}(\tau):=2\pi ic_{n_0}^{2}\eta_{1}(\tau)+
-\sum_{j=1}^{3}n_{j}(n_{j}+1)e_{j}(\tau)\quad\text{with }c_{n_0}^2=i\frac{2n_{0}+1}{2\pi}.
\]
Here and following, we use $\tilde{\tau}$ to denote the variable of $p_{r,s}^{\mathbf{n}}(\cdot)$ and $\tau=\tau(r,s)$ to denote the zero of $Z_{r,s}^{\mathbf{n}_0^-}(\cdot)$.

Consequently,
\begin{equation}\label{fc-45}\wp(p_{r,s}^{\mathbf{n}}(\tilde{\tau});\tilde{\tau})=\frac{1}{c_{n_0}^2
(\tilde{\tau}-\tau(r,s))}-\frac{2h}{c_{n_0}^2}+O((\tilde{\tau}-\tau(r,s))).\end{equation}
Since Lemma \ref{lem-2.3} says that $\wp(p_{r,s}^{\mathbf{n}}(\tilde{\tau});\tilde{\tau})$ depends meromorphically on $(r,s)\in U$, and $\tau(r,s)$, as a simple zero of $Z_{r,s}^{\mathbf{n}_0^-}(\cdot)$, is holomorphic in $(r,s)\in U$, we easily see from (\ref{fc-45}) that $h=h(r,s)$ is also holomorphic in $(r,s)\in U$ and so does $B=B(r,s)$ by (\ref{fc-46}), and
\begin{equation}\label{fc-47}h_r=\frac{B_r+\vartheta_{\mathbf{n}}'(\tau)\tau_r}{-8\pi^2 c_{n_0}^2},\quad h_s=\frac{B_s+\vartheta_{\mathbf{n}}'(\tau)\tau_s}{-8\pi^2 c_{n_0}^2}.\end{equation}

For $\tilde{\tau}$ near $\tau=\tau(r,s)$, it follows from (\ref{fc-44}) that
\begin{align*}
Y_{1;r,s}^{\mathbf{n}}(\tilde{\tau})&=\frac{\partial p_{r,s}^{\mathbf{n}}(\tilde\tau)}{\partial r}=-\tfrac{c_{n_0}}{2}\tau_r(\tilde{\tau}-\tau)^{-\frac{1}{2}}
-\tfrac{3c_{n_0}h}{2}\tau_r(\tilde{\tau}-\tau)^{\frac{1}{2}}\\
&+(c_{n_0}h_r-\tfrac{5c_{n_0}a}{2}\tau_r)(\tilde{\tau}-\tau)^{\frac{3}{2}}
+O((\tilde{\tau}-\tau)^{\frac{5}{2}}),
\end{align*}
\begin{align*}
Y_{2;r,s}^{\mathbf{n}}(\tilde{\tau})&=\frac{\partial p_{r,s}^{\mathbf{n}}(\tilde\tau)}{\partial s}=-\tfrac{c_{n_0}}{2}\tau_s(\tilde{\tau}-\tau)^{-\frac{1}{2}}
-\tfrac{3c_{n_0}h}{2}\tau_s(\tilde{\tau}-\tau)^{\frac{1}{2}}\\
&+(c_{n_0}h_s-\tfrac{5c_{n_0}a}{2}\tau_s)(\tilde{\tau}-\tau)^{\frac{3}{2}}
+O((\tilde{\tau}-\tau)^{\frac{5}{2}}),
\end{align*}
and so
\begin{align*}
\frac{d}{d\tilde\tau}Y_{1;r,s}^{\mathbf{n}}(\tilde{\tau})&=\tfrac{c_{n_0}}{4}\tau_r(\tilde{\tau}-\tau)^{-\frac{3}{2}}
-\tfrac{3c_{n_0}h}{4}\tau_r(\tilde{\tau}-\tau)^{-\frac{1}{2}}\\
&+\tfrac32(c_{n_0}h_r-\tfrac{5c_{n_0}a}{2}\tau_r)(\tilde{\tau}-\tau)^{\frac{1}{2}}
+O((\tilde{\tau}-\tau)^{\frac{3}{2}}),
\end{align*}
\begin{align*}
\frac{d}{d\tilde\tau}Y_{2;r,s}^{\mathbf{n}}(\tilde{\tau})&=\tfrac{c_{n_0}}{4}\tau_s(\tilde{\tau}-\tau)^{-\frac{3}{2}}
-\tfrac{3c_{n_0}h}{4}\tau_s(\tilde{\tau}-\tau)^{-\frac{1}{2}}\\
&+\tfrac32(c_{n_0}h_s-\tfrac{5c_{n_0}a}{2}\tau_s)(\tilde{\tau}-\tau)^{\frac{2}{2}}
+O((\tilde{\tau}-\tau)^{\frac{3}{2}}).
\end{align*}
From here and the Wronskian $W_{\mathbf{n}}(r,s)=\frac{dY_{1;r,s}^{\mathbf{n}}(\tilde\tau)}{d\tilde\tau}Y_{2;r,s}(\tilde\tau)
-\frac{dY_{2;r,s}^{\mathbf{n}}(\tilde\tau)}{d\tilde\tau}Y_{1;r,s}(\tilde\tau)$ is independent of the variable $\tilde{\tau}$, a direct computation gives
\begin{align}\label{fc-49}
W_{\mathbf{n}}(r,s)=c_{n_0}^2(\tau_r h_s-\tau_s h_r)=\frac{\tau_rB_s-\tau_sB_r}{-8\pi^2},
\quad \text{for }\; (r,s)\in U,\end{align}
where we used (\ref{fc-47}) to obtain the second inequality.

Now if (\ref{fc-43}) holds, we have $\tau_rB_s-\tau_sB_r=8\pi^2$, so $W_{\mathbf{n}}(r,s)\equiv -1$ for $(r,s)\in U$ and hence for all $(r,s)\in \mathbb{C}^2\setminus\frac12\mathbb{Z}^2$, because $W_{\mathbf{n}}(r,s)$ is meromorphic in $(r,s)\in \mathbb{C}^2\setminus\frac12\mathbb{Z}^2$. This proves (2)$\Rightarrow$(1).

Conversely,
if $W_{\mathbf{n}}(r,s)\equiv -1$, (\ref{fc-49}) gives $\tau_r B_s-\tau_s B_r=8\pi^2$ for any $(r,s)\in U$, namely (\ref{fc-43}) holds for $(\tau,B)$ in a small neighborhood of $(\tau_0, B_0)$. Since $(\tau_0, B_0)$ is arbitrary, we conclude (\ref{fc-43}) holds for $\varphi_{\mathbf{n}_0^-}$. This proves (1)$\Rightarrow$(2).

(1)$\Leftrightarrow$(3). Again
fix any $(\tau_0,B_0)$ satisfying $Q_{\mathbf{n}_0^+}(B_0;\tau_0)\neq 0$ and let $(r_0,s_0)=\varphi_{\mathbf{n}_0^{+}}(\tau_0,B_0)\notin \frac12\mathbb{Z}^2$ to be the monodromy data of H$(\mathbf{n}_0^{+},B_0,\tau_0)$, i.e. $Z_{r_0,s_0}^{\mathbf{n}_0^+}(\tau_0)=0$.
Then there is a small open neighborhood $U\subset \mathbb{C}^2\setminus\frac12\mathbb{Z}^2$ of $(r_0,s_0)$ such that $Z_{r,s}^{\mathbf{n}_0^+}(\cdot)$ has a unique zero $\tau=\tau(r,s)$ in a small open neighborhood $V\subset \mathbb{H}$ of $\tau_0$ for any $(r,s)\in U$. Let $B=B(r,s)$ be the unique $B$ such that $(r,s)=\varphi_{\mathbf{n}_0^{+}}(\tau(r,s),B(r,s))$ is the monodromy data of H$(\mathbf{n}_0^+,B(r,s),\tau(r,s))$.

Now by Theorem \ref{thm-pole-z1}-(2) and Lemma \ref{lem-3.1}, we still have (\ref{fc-44})-(\ref{fc-46}), where the only different thing is $c_{n_0}^2=-i\frac{2n_{0}+1}{2\pi}$. Therefore, the same argument as \eqref{fc-49} implies
\[\tau_r B_s-\tau_s B_r=-8\pi^2 W_{\mathbf{n}}(r,s),\quad\text{for }\;(r,s)\in U.\]
The rest proof is the same as that of (1)$\Leftrightarrow$(2).
\end{proof}

\subsection{Proof of Theorem \ref{thm-rsbtau}}
Now we can prove Theorem \ref{thm-rsbtau} via an induction approach.

\begin{proof}[Proof of Theorem \ref{thm-rsbtau}] We prove via induction that for any $\mathbf{n}=(n_0,n_1,n_2,n_3)$ with $n_k\geq 0$ and $\max_k n_k\geq 1$,
\begin{equation}d\tau\wedge dB=8\pi^2 dr\wedge ds\label{fc-53}\end{equation}
holds for $\varphi_{\mathbf{n}}: (\tau, B)\to (r,s)$.

{\bf Step 1.} We prove (\ref{fc-53})
for the Lam\'{e} case $\mathbf{n}=(n,0,0,0)$, $n\geq 1$.

In Section 5, we have proved (\ref{fc-53}) for $n=1,2$. From here and Lemma \ref{lem-6.1}, we easily conclude via induction that (\ref{fc-53}) holds for all $n\geq 1$. Furthermore,
\begin{equation}\label{fc-54}
W_{(n,0,0,0)}(r,s)\equiv -1,\quad\text{for any}\; n\geq 0.
\end{equation}

{\bf Step 2.} We prove (\ref{fc-53})
for $\mathbf{n}=(n_0,n_1,0,0)$ with $n_1\geq 1$, $n_0\geq 0$.

Since $y(z)$ solves H$((n_1,0,0,0),B,\tau)$
\[y''(z)=[n_1(n_1+1)\wp(z;\tau)+B]y(z)\]
if and only if $\tilde{y}(z):=y(z+\frac{\omega_1}{2})$ solves H$((0,n_1,0,0),B,\tau)$
\[y''(z)=[n_1(n_1+1)\wp(z+\tfrac{\omega_1}{2};\tau)+B]y(z),\]
so $\varphi_{(0,n_1,0,0)}=\varphi_{(n_1,0,0,0)}$. Together with Step 1, we obtain that (\ref{fc-53}) holds for $(0,n_1,0,0)$, $n_1\geq 1$.

On the other hand, Lemma \ref{lem-3.4} says $p_{r,s}^{(0,n_1,0,0)}(\tau)-\frac{\omega_1}{2}=p_{r+\frac12, s}^{(n_1,0,0,0)}(\tau)$, which implies
\[Y_{1;r,s}^{{(0,n_1,0,0)}}(\tau)=Y_{1;r+\frac12,s}^{{(n_1,0,0,0)}}(\tau),\quad
Y_{2;r,s}^{(0,n_1,0,0)}(\tau)=Y_{2;r+\frac12,s}^{{(n_1,0,0,0)}}(\tau).\]
 From here and (\ref{fc-54}), we obtain
\begin{equation*}
W_{(0,n_1,0,0)}(r,s)=W_{(n_1,0,0,0)}(r+\tfrac12, s)\equiv -1.
\end{equation*}
This together with Lemma \ref{lem-6.1}-(3) implies that (\ref{fc-53}) holds for $(1,n_1,0,0)$. From here and (\ref{fc-53}) holding for $(0,n_1,0,0)$, we easily conclude from Lemma \ref{lem-6.1} that (\ref{fc-53}) holds for all $(n_0,n_1,0,0)$. Furthermore,
\begin{equation}\label{fc-56}
W_{(n_0,n_1,0,0)}(r,s)\equiv -1,\quad\text{for any}\; n_0,n_1\geq 0.
\end{equation}
Clearly the similar argument implies that (\ref{fc-53}) holds
for both $\mathbf{n}=(n_0,0,n_2,0)$ with $n_2\geq 1$ and $\mathbf{n}=(n_0,0,0,n_3)$ with $n_3\geq 1$, and
\begin{equation}\label{fc-56-1}
W_{(n_0,0,n_2,0)}(r,s)=W_{(n_0,0,0,n_3)}(r,s)\equiv -1,\quad\text{for any}\; n_0,n_2,n_3\geq 0.
\end{equation}

{\bf Step 3.} We prove (\ref{fc-53})
for $\mathbf{n}=(n_0,n_1,n_2,0)$ with $n_2\geq 1$, $n_0, n_1\geq 0$.

Since $y(z)$ solves H$((n_1,0,0,n_2),B,\tau)$
\[y''(z)=[n_1(n_1+1)\wp(z;\tau)+n_2(n_2+1)\wp(z+\tfrac{\omega_3}{2};\tau)+B]y(z)\]
if and only if $\tilde{y}(z):=y(z+\frac{\omega_1}{2})$ solves H$((0,n_1,n_2,0),B,\tau)$
\[y''(z)=[n_1(n_1+1)\wp(z+\tfrac{\omega_1}{2};\tau)+
n_2(n_2+1)\wp(z+\tfrac{\omega_2}{2};\tau)+B]y(z),\]
so $\varphi_{(0,n_1,n_2,0)}=\varphi_{(n_1,0,0,n_2)}$. Together with Step 2, we obtain that (\ref{fc-53}) holds for $(0,n_1,n_2,0)$, $n_2\geq 1$.

On the other hand, Lemma \ref{lem-3.4} says $p_{r,s}^{(0,n_1,n_2,0)}(\tau)-\frac{\omega_1}{2}=p_{r+\frac12, s}^{(n_1,0,0,n_2)}(\tau)$. From here and (\ref{fc-56-1}), we obtain
\begin{equation*}
W_{(0,n_1,n_2,0)}(r,s)=W_{(n_1,0,0,n_2)}(r+\tfrac12, s)\equiv -1.
\end{equation*}
This together with Lemma \ref{lem-6.1}-(3) implies that (\ref{fc-53}) holds for $(1,n_1,n_2,0)$. From here and (\ref{fc-53}) holding for $(0,n_1,n_2,0)$, we easily conclude from Lemma \ref{lem-6.1} that (\ref{fc-53}) holds for all $(n_0,n_1,n_2,0)$. Furthermore,
\begin{equation}\label{fc-57}
W_{(n_0,n_1,n_2,0)}(r,s)\equiv -1,\quad\text{for any}\; n_0,n_1,n_2\geq 0.
\end{equation}

{\bf Step 4.} We prove (\ref{fc-53})
for $\mathbf{n}=(n_0,n_1,n_2,n_3)$ with $n_3\geq 1$, $n_0, n_1, n_2\geq 0$.

Since $y(z)$ solves H$((n_3,n_2,n_1,0),B,\tau)$ if and only if $\tilde{y}(z):=y(z+\frac{\omega_3}{2})$ solves H$((0,n_1,n_2,n_3),B,\tau)$,
so $\varphi_{(0,n_1,n_2,n_3)}=\varphi_{(n_3,n_2,n_1,0)}$. Together with Step 3, we obtain that (\ref{fc-53}) holds for $(0,n_1,n_2,n_3)$, $n_3\geq 1$.

On the other hand, Lemma \ref{lem-3.4} says $p_{r,s}^{(0,n_1,n_2,n_3)}(\tau)-\frac{\omega_3}{2}=p_{r+\frac12, s+\frac12}^{(n_3,n_2,n_1,0)}(\tau)$. From here and (\ref{fc-57}), we obtain
\begin{equation*}
W_{(0,n_1,n_2,n_3)}(r,s)=W_{(n_3,n_2,n_1,0)}(r+\tfrac12, s+\tfrac12)\equiv -1.
\end{equation*}
This together with Lemma \ref{lem-6.1}-(3) implies that (\ref{fc-53}) holds for $(1,n_1,n_2,n_3)$. From here and (\ref{fc-53}) holding for $(0,n_1,n_2,n_3)$, we easily conclude from Lemma \ref{lem-6.1} that (\ref{fc-53}) holds for all $(n_0,n_1,n_2,n_3)$. Furthermore,
\begin{equation}\label{fc-58}
W_{(n_0,n_1,n_2,n_3)}(r,s)\equiv -1,\quad\text{for any}\; n_0,n_1,n_2,n_3\geq 0.
\end{equation}
This completes the proof.
\end{proof}

\section{Applications}

In this final section, we give an application of the universal law.
Define $\Delta_j(B)=\Delta_j(B;\tau)$ to be the trace of the monodromy matrix $\rho(\ell_j)$, i.e.
\begin{equation}\label{ex-eq-3}\Delta_1(B;\tau):=2\cos (2\pi s),\quad \Delta_2(B;\tau):=2\cos (2\pi r).
\end{equation}
It is well known that $\Delta_j(B;\tau)$ are \emph{holomorphic} in both $B$ and $\tau$.
\begin{lemma}\label{lem-7.1} For any $(\tau, B)\in \Sigma_{\mathbf{n}}$,
\begin{equation}\label{ex-eq-1}
\Delta_{1,B}:=\tfrac{\partial }{\partial B}\Delta_1=-\tfrac{1}{2\pi}\sin (2\pi s) \tau_r,
\end{equation}
\begin{equation}\label{ex-eq-2}
\Delta_{2,B}:=\tfrac{\partial }{\partial B}\Delta_2=\tfrac{1}{2\pi}\sin (2\pi r) \tau_s.
\end{equation}
\end{lemma}

\begin{proof} Taking derivatives with respect to $r$ and $s$ respectively to $\Delta_1(B;\tau)=2\cos (2\pi s)$, we obtain
\[\begin{pmatrix}
\tau_r & B_r\\
\tau_s & B_s
\end{pmatrix}\begin{pmatrix}
\Delta_{1,\tau} \\
\Delta_{1,B}
\end{pmatrix}=\begin{pmatrix}
0 \\
-4\pi \sin(2\pi  s)
\end{pmatrix},\]
where $\Delta_{1,\tau}={\partial \Delta_1}/{\partial \tau}$.
From here and $\tau_r B_s-\tau_s B_r=8\pi^2$ we easily obtain (\ref{ex-eq-1}). The proof of (\ref{ex-eq-2}) is similar.
\end{proof}

We will see that Lemma \ref{lem-7.1} has interesting applications to the 
algebraic multiplicity of (anti)-periodic eigenvalues for the Hill operator with the DTV potential
\begin{equation}\label{bacq}L_{\mathbf{n}}:=\frac{d^2}{d x^2}-I_{\mathbf{n}}(x;\tau),\quad x\in\mathbb{R}.\end{equation}
Let $B_0$ be any zero of the spectral polynomial $Q_{\mathbf{n}}(B;\tau)$. It follows that
\[\Delta_1(B_0;\tau)=\pm 2,\]
so $B_0$ is a (anti)-periodic eigenvalue of \eqref{bacq} 
with respect to $x\to x+1$.
Denote
\[d(B_0):=\text{ord}_{B_0}(\Delta_1(\cdot;\tau)^2-4)\]
to be the order of $B_0$ as a zero of $\Delta_1(\cdot;\tau)^2-4$.
It is well known (cf. \cite{GW}) that $d(B_0)$ equals to the \emph{algebraic multiplicity} of $B_0$ as a (anti)periodic eigenvalue of  \eqref{bacq}.
For generic $\tau$'s the algebraic multiplicity $d(B_0)=1$. However, for special $\tau$'s the algebraic multiplicity is no longer $1$ and how to compute it remains a long-standing open problem. Here we provide an algorithm of computing the algebraic multiplicity.

\begin{theorem}\label{thm-algebraicM} Recalling the rational function $R_0(B;\tau)\in \mathbb{Q}[\eta_1,e_1,e_2,e_3](B)$ in Theorem \ref{thm-taurs}, there holds
\begin{equation}\label{eq-ex-am}
d(B_0)=2\text{ord}_{B_0}R_0(\cdot;\tau)+2-\text{ord}_{B_0}Q_{\mathbf{n}}(\cdot;\tau),
\end{equation}
namely the algebraic multiplicity $d(B_0)$ can be computed by counting $\text{ord}_{B_0}R_0(\cdot;\tau)$ and $\text{ord}_{B_0}Q_{\mathbf{n}}(\cdot;\tau)$.
\end{theorem}

\begin{proof} In the following argument we omit the notation $\tau$ since it is fixed. Clearly
\[\sin(2\pi s)=\frac12\sqrt{4-\Delta_1(B)^2},\]
and recall Theorem \ref{thm-taurs} that
\[\tau_r=\pi i\frac{R_0(B)}{C}=\pi i \frac{R_0(B)}{\sqrt{Q_{\mathbf{n}}(B)}}.\]
Inserting these into (\ref{ex-eq-1}) leads to
\[\frac{\Delta_{1,B}(B)}{\sqrt{4-\Delta_1(B)^2}}=
\frac{-i}{4}\frac{R_0(B)}{\sqrt{Q_{\mathbf{n}}(B)}}.\]
From here and
\[4-\Delta_1(B)^2\sim (B-B_0)^{d(B_0)},\quad \Delta_{1,B}(B)\sim (B-B_0)^{d(B_0)-1},\]
we easily obtain (\ref{eq-ex-am}).
\end{proof}

\begin{example} The first Lam\'{e} case $\mathbf{n}=(1,0,0,0)$ is simple.
Let us consider the second Lam\'{e} case $\mathbf{n}=(2,0,0,0)$, where we have computed in \eqref{fc-38} that
\[R_0(B;\tau)=-2(B^2+3\eta_1B-\tfrac{3}{2}g_2).\]
Recall \eqref{eq-spec-n=2} that
\begin{align}\label{eq-spec-n=22}
Q_2(B;\tau)&=(B^2-3g_2)(B^3-\tfrac{9}{4}g_2B+\tfrac{27}{4}g_3)\\
&=(B^2-3g_2)\prod_{k=1}^3(B+3e_k).\nonumber
\end{align}
It is easy to prove that
\[\{-3e_{1},-3e_{2},-3e_{3}\} \cap \{(3g_{2})^{1/2},-(3g_{2})^{1/2}\}=\emptyset,
\]
so $-3e_k$'s are always simple zeros of $Q_2(B;\tau)$.

Let $B_0$ be any zero of $Q_2(B;\tau)$. 
There are two cases.

{\bf Case 1.} $B_0=\pm (3g_2)^{1/2}$. It is well known that $g_2(\tau)=0$ if and only if
\[\tau\in\mathfrak{S}:=\left \{  \frac{ae^{\pi i/3}+b}{ce^{\pi i/3}+d}\left \vert
\begin{pmatrix}
a & b\\
c & d
\end{pmatrix}
\in SL(2,\mathbb{Z})\right.  \right \}.
\]
First we consider $\tau\in\mathfrak{S}$, i.e. $g_2(\tau)=0$. Then $B_0=0$ and $\text{ord}_{0}Q_{2}(\cdot;\tau)=2$. Furthermore,
\[R_0(B;\tau)=-2B(B+3\eta_1(\tau)).\]
Since $\eta_1(e^{\pi i/3})=\frac{2\pi}{\sqrt{3}}$ and
\[\eta_{1}\left(\frac{a\tau+b}{c\tau+d}\right)=(c\tau+d)^{2}\eta_1(\tau)-2\pi i c(c\tau+d),\; \begin{pmatrix}
a & b\\
c & d
\end{pmatrix}
\in SL(2,\mathbb{Z})\]
imply that $\eta_1(\tau)\neq 0$ for $\tau\in \mathfrak{S}$, we have $\text{ord}_{0}R_0(\cdot;\tau)=1$. Therefore, we see from (\ref{eq-ex-am}) that
\[d(B_0)=d(0)=2\quad \text{for }\; \tau\in \mathfrak{S}.\]
Next we consider $\tau\notin\mathfrak{S}$, i.e. $g_2(\tau)\neq0$. Then $B_0\neq 0$ and so
$\text{ord}_{B_0}Q_{2}(\cdot;\tau)=1$. Clearly
\[\tfrac{-1}{2}R_0(B;\tau)=(B-B_0)(B+B_0+3\eta_1(\tau))+\tfrac{1}{2}B_0(B_0+6\eta_1(\tau)),\]
so we easily obtain
\[\text{ord}_{B_0}R_0(\cdot;\tau)=\begin{cases} 0\quad \text{if }\;B_0+6\eta_1(\tau)\neq 0,\\
1\quad \text{if }\;B_0+6\eta_1(\tau)= 0.\end{cases}\]
From here and (\ref{eq-ex-am}) it follows that
\[d(B_0)=\begin{cases} 1\quad \text{if }\;B_0+6\eta_1(\tau)\neq 0,\\
3\quad \text{if }\;B_0+6\eta_1(\tau)= 0.\end{cases}\]
In conclusion,
\[d(\pm (3g_2)^{1/2})=\begin{cases} 1 &\text{if }\;\tau\notin\mathfrak{S}, \pm (3g_2)^{1/2}+6\eta_1\neq 0,\\
d(0)=2 &\text{if }\;\tau\in\mathfrak{S},\\
3 &\text{if }\;\pm (3g_2)^{1/2}+6\eta_1= 0.\end{cases}\]
Recently, we proved in \cite{CL-E2} that there are infintely many $\tau$'s such that $12\eta_1(\tau)^2-g_2(\tau)=0$, so for such $\tau$'s, either $d((3g_2)^{1/2})=3$ or $d(-(3g_2)^{1/2})=3$.

{\bf Case 2.} $B_0=-3e_k(\tau)$. Then $\text{ord}_{-3e_k}Q_{2}(\cdot;\tau)=1$ and
\[\tfrac{-1}{2}R_0(B;\tau)=(B+3e_k)(B-3e_k+3\eta_1)-9\pi i e_k'(\tau),\]
where we used
\[
e_{k}^{\prime}(\tau)=\frac{i}{\pi}\Big[  \frac{1}{6}g_{2}(\tau)+\eta
_{1}(\tau)e_{k}(\tau)-e_{k}(\tau)^{2}\Big].
\]
So if $e_k'(\tau)\neq 0$, we have $\text{ord}_{-3e_k}R_0(\cdot;\tau)=0$ and so $d(-3e_k)=1$. If $e_k'(\tau)=0$ and $2e_k-\eta_1\neq 0$, we have  $\text{ord}_{-3e_k}R_0(\cdot;\tau)=1$ and so $d(-3e_k)=3$. If $e_k'(\tau)=0$ and $2e_k-\eta_1= 0$, or equivalently
\begin{equation}\label{eq-ex-ek}
3\eta_1^2+2g_2=0,\quad 6e_k^2+g_2=0,
\end{equation}
then $\text{ord}_{-3e_k}R_0(\cdot;\tau)=2$ and so $d(-3e_k)=5$. In conclusion,
\[d(-3e_k)=\begin{cases} 1 &\text{if }\;e_k'(\tau)\neq 0,\\
3 &\text{if }\;e_k'(\tau)=0, 2e_k-\eta_1\neq 0\\
5 &\text{if }\;e_k'(\tau)=0, 2e_k-\eta_1= 0.\end{cases}\]
Recently, we proved in \cite{CL=CMP} that there are infintely many $\tau$'s such that $e_k'(\tau)=0$, so for such $\tau$'s, $d(-3e_k)\in\{3,5\}$.
Whether there exist $\tau$ satisfying (\ref{eq-ex-ek}) remains as an interesting open problem.
\end{example}

\subsection*{Acknowledgements} Z. Chen was supported by NSFC (No. 12222109, 12071240).

\end{document}